\documentclass[12pt]{amsart}
\usepackage{a4wide,enumerate,color,graphicx}
\usepackage{amsmath,amssymb,amsfonts}
\usepackage{scrextend} 
\usepackage{esint} 
\allowdisplaybreaks

\let\pa\partial
\let\na\nabla
\let\eps\varepsilon
\newcommand{\N}{{\mathbb N}}
\newcommand{\R}{{\mathbb R}}

\newcommand{\diver}{\operatorname{div}}

\newtheorem{theorem}{Theorem}
\newtheorem{lemma}[theorem]{Lemma}

\usepackage{xcolor}

\begin{document}

\title[A degenerate diffusion system]{Existence analysis of a degenerate 
diffusion system for heat-conducting fluids}

\author[G. Favre]{Gianluca Favre}
\address{Fakult\"at f\"ur Mathematik, Universit\"at Wien, Oskar-Morgenstern-Platz 1,
1090 Wien, Austria}
\email{gianluca.favre@univie.ac.at}

\author[A. J\"ungel]{Ansgar J\"ungel}
\address{Institute for Analysis and Scientific Computing, Vienna University of
	Technology, Wiedner Hauptstra\ss e 8--10, 1040 Wien, Austria}
\email{juengel@tuwien.ac.at}

\author[C. Schmeiser]{Christian Schmeiser}
\address{Fakult\"at f\"ur Mathematik, Universit\"at Wien, Oskar-Morgenstern-Platz 1,
1090 Wien, Austria}
\email{christian.schmeiser@univie.ac.at}

\author[N. Zamponi]{Nicola Zamponi}
\address{University of Mannheim, School of Business Informatics and Mathematics,
B6, 28, 68159 Mannheim, Germany}
\email{nzamponi@mail.uni-mannheim.de}

\date{\today}

\thanks{The first three authors acknowledge partial support from the FWF,
the Austrian Science Fund (FWF), grants F65 and W1245. 
The second author has been additionally supported by the grants P30000 and P33010
of the FWF. The fourth author acknowledges support from the Alexander von 
Humboldt Foundation}

\begin{abstract}
The existence of global weak solutions to a parabolic energy-transport 
system in a bounded domain with no-flux boundary conditions is proved. The model
can be derived in the diffusion limit from a kinetic equation 
with a linear collision operator involving
a non-isothermal Maxwellian. The evolution of the local temperature is governed by 
a heat equation with a source term that depends on the energy of the distribution 
function. The limiting model consists of cross-diffusion equations with an entropy
structure. The main difficulty is the nonstandard degeneracy, i.e., ellipticity 
is lost when the fluid density or temperature vanishes. 
The existence proof is based on a priori estimates coming from 
the entropy inequality and the $H^{-1}$ method and on techniques
from mathematical fluid dynamics (renormalized formulation, div-curl lemma).
\end{abstract}

\keywords{Cross-diffusion, energy transport, weak solutions, renormalized equation,
compensated compactness.}

\subjclass[2000]{35K51, 35K65, 35Q79.}

\maketitle


\section{Introduction}

This paper is concerned with the global existence analysis of a degenerate
diffusion system governing the evolution of the particle density 
$\rho(x,t)$ and temperature $\theta(x,t)$:
\begin{equation}\label{1.eq}
  \pa_t\rho = \Delta(\rho\theta), \quad 
	\pa_t E	= \Delta\bigg(\theta + \frac{5}{2}\rho\theta^2\bigg)
	\quad\mbox{in }\Omega,\ t>0,
\end{equation}
where $E=\theta + \frac32\rho\theta$ is the energy density, supplemented by
no-flux boundary and initial conditions,
\begin{align}
  \na(\rho\theta)\cdot\nu=\na\bigg(\theta+\frac{5}{2}\rho\theta^2\bigg)\cdot\nu=0
	&\quad\mbox{on }\pa\Omega,\ t>0, \label{1.bc} \\ 
	\rho(0)=\rho^0,\ E(0)=E^0:=\theta^0+\frac32\rho^0\theta^0 &\quad\mbox{in }\Omega, \label{1.ic}
\end{align}
and $\Omega\subset\R^3$ is a bounded domain. The equations describe a rarefied gas
with thermalizing collisions. They can be formally derived from a collisional kinetic
equation, coupled to a heat equation for the background temperature 
governed by a Fourier law. We refer to Section \ref{sec.model} for modeling details.

A major difficulty of system \eqref{1.eq} is the
derivation of suitable a priori estimates.
This issue will be tackled by exploiting the entropy structure of the system.
This means that equations \eqref{1.eq} can be written in the 
cross-diffusion form
\begin{equation}\label{1.eq2}
  \pa_t \vec{u} = \diver(M\na \vec{q}),
\end{equation}
where 
$$
  \vec{u} = \begin{pmatrix} \rho \\ E \end{pmatrix}, \quad
	\vec{q} = \begin{pmatrix} \log(\rho/\theta^{3/2})+\frac52 \\ -1/\theta
	\end{pmatrix}, \quad
  M = \begin{pmatrix}
	\rho\theta & \tfrac{5}{2}\rho\theta^2 \\
	\tfrac{5}{2}\rho\theta^2 & \theta^2(1+\tfrac{35}{4}\rho\theta)
	\end{pmatrix}.
$$
The so-called Onsager matrix $M$ is symmetric and positive semidefinite. 
However, $M$ becomes indefinite
when $\rho=0$ or $\theta=0$, showing that \eqref{1.eq2} is of degenerate type. 
The Gibbs free energy
\begin{equation}\label{gibbs}
  G = \rho\theta\log\frac{\rho}{\theta^{3/2}} + \frac32\rho\theta
	- \theta(\log\theta-1),
\end{equation}
defines the 
\begin{itemize}
\item chemical potential $\mu=\pa G/\pa\rho=\theta(\log(\rho/\theta^{3/2})+\frac52)$, 
\item the (mathematical) entropy $h=\pa G/\pa\theta=\rho\log(\rho/\theta^{3/2})
-\log\theta$, and
\item the energy density $E=G-\theta\pa G/\pa\theta=(1+\frac32\rho)\theta$. 
\end{itemize}
We reveal the formal gradient-flow structure for \eqref{1.eq2} by
defining the thermo-chemical potential $\phi=\pa h/\pa\rho=\mu/\theta$ and the 
negative inverse temperature $\pa h/\pa E=-1/\theta$ (interpreting $h$ as a function
of $(\rho,E)$) such that
$$
  \pa_t(\rho,E)^T - \diver(M\na Dh) = 0,
$$
where $Dh$ is the vector with components $\pa h/\pa\rho$ and
$\pa h/\pa E$. Furthermore, the entropy $h$ is a Lyapunov functional along
solutions to \eqref{1.eq2}: 
$$
  \frac{d}{dt}\int_\Omega h dx 
	= \int_\Omega\bigg(\frac{\pa h}{\pa\rho}\pa_t\rho 
	+ \frac{\pa h}{\pa E}\pa_t E\bigg)dx
	= -\int_\Omega \na(Dh)^T M\na Dh dx \le 0,
$$
since $M$ is positive semidefinite. In particular, we obtain a priori estimates
for $\na(Dh)^T M$ $\times\na Dh$ in $L^1(\Omega)$ from which we conclude gradient
estimates for $\sqrt{\rho\theta}$ and $\log\theta$ in $L^2(\Omega)$ (see below).

Still, this approach is not sufficient. Indeed, because of the degeneracy at 
$\theta=0$, we cannot expect to achieve any control on the gradient of $\rho$,
and moreover, the bounds from the entropy estimate are not sufficient to conclude.
Our idea, detailed below, is to apply well-known tools from mathematical fluid 
dynamics like $H^{-1}$ estimates and compensated compactness. The originality
of this work consists in the combination of these tools and entropy methods,
which allows us to treat non-standard degeneracies.


\subsection{State of the art}

Equations \eqref{1.eq} belong to the class of energy-transport models which
have been investigated particularly in semiconductor theory \cite{Jue09}.
The first energy-transport model for semiconductors was presented by
Stratton \cite{Str62}. First existence results were concerned with models
with very particular diffusion coefficients (being not of the form \eqref{1.eq})
\cite{AlRo17,AlXi94} or with uniformly positive definite diffusion matrices 
\cite{DGJ97}. Existence results for
physically more realistic diffusion coefficients were shown in \cite{ChHs03},
but only for situations close to equilibrium. A degenerate energy-transport
system with a simplified temperature equation was analyzed in \cite{JPR13}.
Energy-transport models do not only appear in semiconductor theory. For instance,
they have been used to model self-gravitating particle clouds \cite{BiNa02} and
the dynamics in optical lattices \cite{BrJu18}.

In \cite{ZaJu15}, the global existence of weak solutions to the model
\begin{equation}\label{et}
  \pa_t\rho = \Delta(\rho\theta), \quad 
	\pa_t(\rho\theta) = \frac53\Delta(\rho\theta^2)
\end{equation}
in a bounded domain $\Omega$ with no-flux boundary conditions was proved. 
At first glance, equations \eqref{1.eq} look simpler than \eqref{et} because
of the additional diffusion in the energy equation. 
However, the ideas in \cite{ZaJu15} cannot be easily applied to \eqref{1.eq}. 
Indeed, the key idea in \cite{ZaJu15}
was to introduce the variables $u=\rho\theta$ and $v=\rho\theta^2$ and to
apply the Stampacchia trunction method to a time-discretized version of
\begin{equation}\label{1.uv}
  \pa_t\bigg(\frac{u^2}{v}\bigg) = \Delta u, \quad \pa_t u = \frac53\Delta v.
\end{equation}
The functionals $\int_\Omega\rho^2\theta^b dx$ turn out to be
Lyapunov functionals along solutions to \eqref{1.uv}
for suitable values of $b\in\R$, leading to uniform gradient estimates. 
However, the additional term in the energy equation of \eqref{1.eq}
complicates the derivation of a priori estimates.
Thus, the proof in \cite{ZaJu15} seems to be rather specific to 
system \eqref{et} and is not generalizable. Our idea is to treat \eqref{1.eq} 
by combining entropy methods and 
tools from mathematical fluid dynamics, which may be also applied to
other cross-diffusion systems. 


\subsection{Mathematical key ideas}

As explained before, the first key idea is to exploit, in contrast to \cite{ZaJu15}, 
the entropy structure of \eqref{1.eq}.
Indeed, recalling the mathematical entropy density 
\begin{equation}\label{1.h}
  h(\rho,\theta) = \rho\log\frac{\rho}{\theta^{3/2}}-\log\theta \quad\mbox{for }
	\rho,\,\theta>0,
\end{equation}
a formal computation (which is made rigorous for an approximate scheme; see
\eqref{ei2}) gives the entropy dissipation equation
$$
  \frac{d}{dt}\int_\Omega h(\rho,\theta)dx
	+ \int_\Omega\bigg(2\big|\na\sqrt{\rho\theta}\big|^2 
	+ |\na\log\theta|^2\bigg(1+\frac{5}{2}\rho\theta\bigg)\bigg)dx = 0 \,,
$$
which provides $H^1(\Omega)$ estimates for $\sqrt{\rho\theta}$ and $\log\theta$.
Moreover, this estimate implies that $\theta>0$ a.e.\ (but not $\rho>0$).

Clearly, the entropy estimates are not sufficient to pass to the de-regularization
limit in the approximate scheme. Further bounds
are derived from the $H^{-1}(\Omega)$ method, i.e., we use basically
$(-\Delta)^{-1}\rho$ and $(-\Delta)^{-1}E$, respectively, 
as test functions in the weak formulation
of \eqref{1.eq} (second key idea). This method gives estimates for 
$$
  \int_\Omega\rho^2\theta dx   \qquad\mbox{and}\qquad
	\int_\Omega\bigg(\theta+\frac{5}{2}\rho\theta^2\bigg)
	\bigg(\theta+\frac32\rho\theta\bigg)dx \,.
$$
Combining these bounds with those coming from the entropy inequality 
and the conservation laws leads to
estimates for $\na(\rho\theta)=\sqrt{\rho\theta}\na\sqrt{\rho\theta}$, 
$\na\theta=\theta\na\log\theta$ and consequently for $E$ in $W^{1,1}(\Omega)$.
Moreover, $\pa_t E$ is bounded in some dual Sobolev space. This allows us to
apply the Aubin--Lions lemma to $E$. Unfortunately, we do not obtain 
gradient estimates for $\rho$. 

To overcome this issue, we use tools from mathematical fluid dynamics
(third key idea).
Let $(\rho_\delta,\theta_\delta)$ be approximate solutions to \eqref{1.eq}
(in a sense made precise in Section \ref{sec.ex}). 
First, we write the mass balance equation in the renormalized form
$$
  \pa_t f(\rho_\delta) - \diver(f'(\rho_\delta)\na(\rho_\delta\theta_\delta))
	= -f''(\rho_\delta)\na\rho_\delta\cdot\na(\rho_\delta\theta_\delta)
$$
in the sense of distributions for smooth functions $f$ with bounded derivatives. 
Let $g$ another smooth function with bounded derivatives and introduce the
vectors
$$
  U_\delta = \big(f(\rho_\delta),-f'(\rho_\delta)
	\na(\rho_\delta\theta_\delta)\big), \quad
	V_\delta = \big(g(\theta_\delta),0,0,0\big).
$$
We deduce from the properties of $f$ and $g$ and the a priori estimates that
$\diver_{(t,x)}U_\delta$ and $\operatorname{curl}_{(t,x)}V_\delta$ are uniformly
bounded in $L^1(\Omega\times(0,T))$ and hence relatively compact in
$W^{-1,r}(\Omega)$ for some $r>1$. The div-curl lemma implies that
$\overline{U_\delta\cdot V_\delta}=\overline{U_\delta}\cdot\overline{V_\delta}$
a.e., where the bar denotes the weak limit of the corresponding sequence.
Thus, $\overline{f(\rho_\delta)g(\theta_\delta)}=\overline{f(\rho_\delta)}\;
\overline{g(\theta_\delta)}$ a.e. A truncation procedure yields that
$\overline{\rho_\delta\theta_\delta}=\rho\theta$, where $\rho$ and $\theta$
are the weak limits of $(\rho_\delta)$ and $(\theta_\delta)$, respectively.
As $(E_\delta)$ converges strongly, by the Aubin--Lions lemma,
we are able to prove that $\theta_\delta\to\theta$ and
eventually $\rho_\delta\to\rho$ a.e. These limits allow us to identify the
weak limits and to pass to the limit $\delta\to 0$ in the approximate equations.
The approximate scheme contains additional terms which need to be treated carefully
such that our arguments are more technical than presented here. In fact,
we need three approximation levels; see Section \ref{sec.ex} for details.


\subsection{Main result}

Our main result is as follows:

\begin{theorem}[Existence of weak solutions]\label{thm.ex}
Let $\Omega\subset\R^3$ be a bounded domain with $\pa\Omega\in C^{1,1}$.
Let $\rho^0$, $\theta^0\in L^1(\Omega)$ satisfy $\rho^0\ge 0$, $\theta^0\ge 0$ in 
$\Omega$ and $\rho^0\theta^0$, $h(\rho^0,\theta^0)\in L^1(\Omega)$, where
$h$ is defined in \eqref{1.h}. Let $T>0$ and $\Omega_T=\Omega\times(0,T)$.
Then there exist $\rho$, $\theta\in L^\infty(0,T;L^1(\Omega))$ such that
\begin{align*}
  & \rho\log\rho\in L^\infty(0,T; L^1(\Omega)),
	\quad E = \theta + \frac32 \rho\theta \in 
	L^\infty(0,T; L^1(\Omega))\cap L^2(\Omega_T), \\
  & \sqrt{\rho\theta},\,\log\theta\in L^2(0,T; H^1(\Omega)),
  \quad \rho\theta^2\in L^{3/2}(\Omega_T), \\
	& \pa_t\rho\in L^{4/3}(0,T;W^{1,4}(\Omega)'), \quad
	\pa_t E\in L^{6/5}(0,T;W^{2,4}(\Omega)');
\end{align*}
it holds that $\rho\ge 0$ and $\theta>0$ a.e.\ in $\Omega_T$; 
$(\rho,\theta)$ is a weak solution to \eqref{1.eq}--\eqref{1.ic} in the sense
\begin{align}
  \int_0^T\langle\pa_t\rho,\psi_1\rangle dt
  + \frac32\int_0^T\int_\Omega\na(\rho\theta)\cdot\na\psi_1 dxdt &= 0, \label{weak.1} \\
  \int_0^T\langle\pa_t E,\psi_2\rangle dt
  -\int_0^T\int_\Omega\bigg(\theta+\frac{5}{2}\rho\theta^2\bigg)
  \Delta\psi_2 dxdt &=0 ,\label{weak.2} 
\end{align}
for any test functions $\psi_1\in L^4(0,T; W^{1,4}(\Omega))$,
$\psi_2\in L^{6}(0,T; W^{2,4}(\Omega))$; and the initial data \eqref{1.ic}
is satisfied in the sense of $W^{1,4}(\Omega)'$ and $W^{2,4}(\Omega)'$,
respectively. Moreover, the total mass and energy are preserved:
$$
  \int_\Omega\rho(t)dx = \int_\Omega\rho^0 dx, \quad
	\int_\Omega E(t)dx = \int_\Omega E^0 dx\quad\mbox{for }t\ge 0.
$$
\end{theorem}

The paper is organized as follows. Equations \eqref{1.eq} are formally derived
from a relaxation-time kinetic model in Section \ref{sec.model}, while
the proof of Theorem \ref{thm.ex} is presented in Section \ref{sec.ex}.


\section{Formal derivation from a kinetic model}
\label{sec.model}

We consider a gas which is rarefied enough such that collisions between gas particles 
can be neglected, but there are thermalizing collisions at a fixed rate 
with a nonmoving background. This is modeled by sampling post-collisional velocities 
from a Maxwellian distribution with zero mean velocity and with 
the background temperature, which is determined from the assumptions of energy 
conservation as well as heat transport in the background governed by the 
Fourier law. These assumptions lead to the equations
\begin{align}
  \label{kin resc}
  \eps^2 \pa_t f_\eps + \eps v\cdot \na f_\eps 
	&= \rho_\eps M(\theta_\eps) - f_\eps, \\
  \label{heat resc}
  \eps^2 (\pa_t \theta_\eps - \Delta \theta_\eps) 
	&= \frac12\int_{\R^3} |v|^2 (f_\eps-\rho_\eps M(\theta_\eps))dv,
\end{align}
which are written in dimensionless form with a diffusive macroscopic scaling with 
the scaled Knudsen number $0<\eps\ll 1$. The gas is described by the 
distribution function $f_\eps(x,v,t)$ with the velocity $v\in\R^3$, and the 
temperature of the background is $\theta_\eps(x,t)$. The gradient and Laplace 
operators are meant with respect to the position variable $x$, 
and the Maxwellian is given by
\begin{equation}
\label{Gaussian}
   M(\theta;v) = \frac{1}{(2\pi \theta)^{3/2}} 
	\exp\bigg(-\frac{|v|^2}{2\theta}\bigg).
\end{equation}
Finally, the position density of the gas is defined by 
$$
  \rho_\eps(x,t) = \int_{\R^3} f_\eps(x,v,t)dv. 
$$
The right-hand side of the heat equation \eqref{heat resc} has been chosen such that 
the sum of the kinetic energy of the gas and the thermal energy
of the background is conserved. 
In \cite{FaScPi20}, the energy-transport system \eqref{1.eq} has been derived 
formally from \eqref{kin resc}--\eqref{heat resc} in the macroscopic limit 
$\eps\to 0$. We repeat the argument here for completeness. 

In the computations, the moments of the Maxwellian up to order 4 will be needed:
\begin{equation}\label{moments}
\begin{aligned}
  & \int_{\R^3} M(\theta;v)dv = 1,\quad 
	\int_{\R^3} vM(\theta;v)dv =  \int_{\R^3} v|v|^2M(\theta;v)dv = 0, \\
  & \int_{\R^3} v_i v_j M(\theta;v)dv = \theta \delta_{ij},\quad  
	\int_{\R^3} v_i v_j|v|^2 M(\theta;v)dv = 5\theta^2\delta_{ij},
\end{aligned}
\end{equation}
where $v_i$, $v_j$ denote the components of $v$ ($i,j=1,2,3$). From 
\eqref{kin resc}--\eqref{heat resc}, the local conservation laws for mass and energy,
\begin{align*}
  & \pa_t \rho_\eps + \diver \bigg(\frac{1}{\eps}\int_{\R^3} vf_\eps dv\bigg) = 0,\\
  & \pa_t \bigg(\theta_\eps + \frac12\int_{\R^3} |v|^2 f_\eps dv\bigg) 
	+ \diver \bigg( \frac{1}{2\eps}\int_{\R^3} v|v|^2 f_\eps dv 
	- \na\theta_\eps\bigg) = 0,
\end{align*}
can be derived by integration of \eqref{kin resc} with respect to $v$ and, 
respectively, by integration of \eqref{kin resc} against $|v|^2/2$ and adding to 
\eqref{heat resc}.

In a formal convergence analysis, we assume $f_\eps\to f$, $\rho_\eps \to \rho$, 
and $\theta_\eps\to\theta$ as $\eps\to 0$ and deduce from \eqref{kin resc}
that $f=\rho M(\theta)$. With \eqref{moments}, we obtain for the kinetic energy 
density
$$
   \lim_{\eps\to 0} \frac12\int_{\R^3} |v|^2 f_\eps dv = \frac32\rho\theta.
$$
The limit of the mass flux is obtained by multiplication of \eqref{kin resc} by 
$v/\eps$, integration with respect to $v$, and passing to the limit, using again
\eqref{moments}:
\begin{align*}
  \lim_{\eps\to 0} \bigg(\frac{1}{\eps}\int_{\R^3} vf_\eps dv\bigg)  
	&= -\int_{\R^3} v (v\cdot\nabla(\rho M(\theta;v)))dv \\
  &= -\diver \bigg( \rho \int_{\R^3} v\otimes v M(\theta;v)dv\bigg) 
	= - \na(\rho\theta).
\end{align*}
Analogously, we compute the flux of the kinetic energy,
\begin{align*}
  \lim_{\eps\to 0} \bigg(\frac{1}{2\eps}\int_{\R^3} v_i|v|^2 f_\eps dv\bigg)  
	&= -\frac12\int_{\R^3} v_i|v|^2 (v\cdot\nabla(\rho M(\theta;v)))dv \\
  &= - \frac12\sum_{j=1}^3\frac{\pa}{\pa x_j}\bigg( \rho \int_{\R^3} v_iv_j 
	|v|^2 M(\theta;v)dv\bigg) 
	= - \frac52\frac{\pa}{\pa x_i}(\rho\theta^2)
\end{align*}
for $i=1,2,3$. 
Using these results in the limits of the conservation laws leads to \eqref{1.eq}.


\section{Proof of Theorem \ref{thm.ex}}\label{sec.ex}

We approximate equations \eqref{1.eq} in the following way.
The time derivative is replaced by the implicit Euler discretization
with parameter $\tau>0$. This is needed to avoid issues related to 
the time regularity.
A higher-order $H^4$ regularization for $\phi=\pa h/\pa\rho$ in the mass 
balance equation with parameter $\eps>0$
gives $H^2(\Omega)$ regularity and compactness in
$W^{1,4}(\Omega)$. Furthermore, $H^2(\Omega)$ and $W^{1,4}(\Omega)$ regularizations 
for $\log\theta$ with the same parameter 
are added to the energy balance equation. The $W^{1,4}(\Omega)$
regularization is needed to derive estimates when using both $\log\theta$ and 
$-1/\theta$ as test functions in \eqref{1.eq}.
Furthermore, we add an additional $H^2(\Omega)$ regularization for $\phi$ in the 
mass balance equation with parameter $\delta>0$, which removes the degeneracy 
of the diffusion matrix $M$ in \eqref{1.eq2}.
Finally, we add the artificial heat flux $\Delta\theta^{3}$ 
in the energy density equation with the same parameter $\delta$ to obtain 
gradient estimates for the temperature, and we add the term $\theta^{-N}\log\theta$
for some $N>0$ to achieve an estimate for $\theta^{-(N+1)}$. 

After having proved the existence of solutions to the approximate problem and some
a priori estimates coming from the entropy inequality, we perform the limits
$\eps\to 0$, $\tau\to 0$, and $\delta\to 0$ (in this order). 

\subsection{Solution of the approximate problem}

We wish to solve a system which approximates \eqref{1.eq} and is formulated
in the variables $\phi$ and $w=\log\theta$, similarly as in \eqref{1.eq2}.
We interpret $\rho$ and $E=\theta(1+\frac32\rho)$ as functions of $(\phi,w)$, i.e.
$$
  \rho(\phi,w) = \exp\bigg(\phi+\frac32w-\frac52\bigg), \quad
	E(\rho,w) = \bigg(1+\frac32\rho(\phi,w)\bigg)\exp(w).
$$
In this notation, the diffusion coefficients become
\begin{equation}\label{2.M}
  M_{11} = \rho e^w, \quad
	M_{12} = \frac{5}{2}\rho e^{2w}, \quad
  M_{22} = e^{2w}\bigg(1+\frac{35}{4}\rho e^w\bigg).
\end{equation}

Let $T>0$ and let the approximation parameters $\tau>0$ (such that $T/\tau\in\N$), 
$\eps>0$, and $\delta>0$ be given. Furthermore, let $0<N<5$ be a number needed for the
approximation $\theta^{-N}\log\theta$ in the energy balance equation.

We wish to find $(\phi^k,w^k)\in H^2(\Omega;\R^2)$ such that, 
with $\rho^k=\rho(\phi^k,w^k)$, $E^k=E(\rho^k,w^k)$,
\begin{align}
  0 &= \frac{1}{\tau}\int_\Omega(\rho^k-\rho^{k-1})\psi_1 dx
	+ \int_\Omega(M_{11}^k\na\phi^k + M_{12}^ke^{-w^k}\na w^k)\cdot\na\psi_1 dx 
	\label{ap1} \\
	&\phantom{xx}{}+ \eps\int_\Omega D^2\phi^k:D^2\psi_1 dx
	+ \delta\int_\Omega(\na\phi^k\cdot\na\psi_1 + \phi^k\psi_1) dx, \nonumber \\
  0 &= \frac{1}{\tau}\int_\Omega(E^k-E^{k-1})\psi_2 dx 
	+ \int_\Omega(M_{12}^k\na\phi^k + M_{22}^ke^{-w^k}\na w^k)\cdot\na\psi_2 dx
	\label{ap2} \\
	&\phantom{xx}{}+ \eps\int_\Omega e^{w^k}\big(D^2 w^k:D^2\psi_2 
	+ |\na w^k|^2\na w^k\cdot\na\psi_2\big)dx 
	+ \eps\int_\Omega(1+e^{w^k})w^k\psi_2 dx \nonumber \\
  &\phantom{xx}{}
	+ \delta\int_\Omega e^{3w^k}\na w^k\cdot\na\psi_2 dx
	+ \delta\int_\Omega e^{-N w^k}w^k\psi_2 dx
	\nonumber
\end{align}
for all $(\psi_1,\psi_2)\in H^2(\Omega;\R^2)$, and $M_{ij}^k$ are given by
\eqref{2.M} with $(\rho,w)$ replaced by $(\rho^k,w^k)$. The existence of solutions
to \eqref{ap1}--\eqref{ap2} is shown in two steps.

{\em Step 1: solution of the linearized approximated problem.}
In the following, we drop the superindex $k$.
Let $(\widetilde\phi,\widetilde w)\in W^{1,4}(\Omega;\R^2)$ be given and set
$\widetilde\rho=\rho(\widetilde\phi,\widetilde w)$, 
$\widetilde E=E(\widetilde\phi,\widetilde w)$.
We wish to find $(\phi,w)\in H^2(\Omega;\R^2)$ such that 
\begin{equation}\label{lm}
  a_1(\phi,\psi_1) = \sigma F_1(\psi_1), \quad a_2(w,\psi_2) = \sigma F_2(\psi_2)
\end{equation}
for all $(\psi_1,\psi_2)\in H^2(\Omega;\R^2)$, where $\sigma\in[0,1]$ and
\begin{align*}
  a_1(\phi,\psi_1) &= \eps\int_\Omega D^2\phi:D^2\psi_1 dx
	+ \delta\int_\Omega(\na\phi\cdot\na\psi_1 + \phi\psi_1) dx, \\
	a_2(w,\psi_2) &= \eps\int_\Omega e^{\widetilde w}\big(D^2 w:D^2\psi_2 
	+ |\na\widetilde w|^2\na w\cdot\na\psi_2\big)dx 
	+ \eps\int_\Omega(1+e^{\widetilde w})w\psi_2 dx \\
	&\phantom{xx}{}+ \delta\int_\Omega e^{3\widetilde w}\na w
	\cdot\na\psi_2 dx + \delta\int_\Omega e^{-N \widetilde{w}}w\psi_2 dx, \\
  F_1(\psi_1) &= -\frac{1}{\tau}\int_\Omega(\widetilde\rho-\rho^{k-1})\psi_1 dx
	- \int_\Omega(\widetilde M_{11}\na\widetilde\phi + \widetilde M_{12}
	e^{-\widetilde w}\na\widetilde w)\cdot\na\psi_1 dx\\
  F_2(\psi_2) &= -\frac{1}{\tau}\int_\Omega(\widetilde E-E^{k-1})\psi_2 dx 
	- \int_\Omega(\widetilde M_{12}\na\widetilde\phi + \widetilde M_{22}
	e^{-\widetilde w}\na\widetilde w)\cdot\na\psi_2 dx,
\end{align*}
where $\widetilde{M}_{ij}$ is given by \eqref{2.M} with $(\rho,w)$ replaced by
$(\widetilde\rho,\widetilde w)$.
The bilinear forms $a_1$ and $a_2$ are coercive on $H^2(\Omega)$ since,
by the generalized Poincar\'e inequality \cite[Chap.~2, Sect.~1.4]{Tem97},
\begin{align*}
  a_1(\phi,\phi) &= \eps\int_\Omega |D^2\phi|^2 dx 
	+ \delta\int_\Omega (|\na\phi|^2 + \phi^2) dx
	\ge \min\{\eps,\delta\}\|\phi\|_{H^2(\Omega)}^2, \\
  a_2(w,w) &\ge \eps\int_\Omega(C|D^2 w|^2 + w^2)dx \ge \eps C\|w\|_{H^2(\Omega)}^2
\end{align*}
for some constant $C>0$.
The linear forms $F_1$ and $F_2$ are continuous on $H^2(\Omega)$ since,
by the continuous embedding $W^{1,4}(\Omega)\hookrightarrow L^\infty(\Omega)$,
$\widetilde\phi$ and $\widetilde w$ are $L^\infty(\Omega)$ functions such that
$\widetilde\rho$, $\widetilde E\in L^\infty(\Omega)$ too. 
The Lax--Milgram lemma implies the
existence of a unique solution $(\phi,w)$ to \eqref{lm} such that
$\rho=\rho(\phi,w)>0$ and $E=E(\phi,w)>0$. This defines the fixed-point operator 
$S:W^{1,4}(\Omega;\R^2)\times[0,1]\to W^{1,4}(\Omega;\R^2)$,
$S(\widetilde\phi,\widetilde w,\sigma)=(\phi,w)$, where $(\phi,w)$ solves \eqref{lm}.

{\em Step 2: solution to the approximate problem.}
We wish to apply the Leray--Schauder fixed-point theorem. It holds that
$S(\widetilde\phi,\widetilde w,0)=0$. Standard arguments show that 
$S:W^{1,4}(\Omega;\R^2)\to H^2(\Omega;\R^2)$ is continuous.
Since $H^2(\Omega;\R^2)$ is compactly embedded
into $W^{1,4}(\Omega;\R^2)$, $S:W^{1,4}(\Omega;\R^2)\to W^{1,4}(\Omega;\R^2)$ 
is compact. It remains to show that 
there exists a uniform bound in $W^{1,4}(\Omega;\R^2)$ for all fixed points. 

Let $\sigma\in(0,1]$ and let $(\phi,w)$ be a fixed point of $S(\cdot,\cdot,\sigma)$. 
It is a solution to \eqref{ap1}--\eqref{ap2}
with $\phi=\phi^k$, $w=w^k$, $\rho=\rho^k$, and $E=E^k$.
We use the test functions $\psi_1=\phi$ and $\psi_2=1-e^{-w}$ in \eqref{ap1} and
\eqref{ap2}, respectively, and add both equations. 
(We use $1-e^{-w}$ instead of $-e^{-w}$ as a test function in order to be able to 
treat the term $\eps\int_\Omega (1+e^w)w\psi_2 dx$ and to obtain the entropy
and energy balance in one single equation.) Then
\begin{align}
  0 &= \frac{\sigma}{\tau}\int_\Omega\big((\rho-\rho^{k-1})\phi 
	+ (E-E^{k-1})(1-e^{-w})\big)dx \nonumber \\
	&\phantom{xx}{}+ \int_\Omega\big(M_{11}|\na\phi|^2 + 2M_{12}e^{-w}\na\phi\cdot\na w
	+ M_{22}e^{-2w}|\na w|^2\big)dx	\nonumber \\
	&\phantom{xx}{}
	+ \eps\int_\Omega e^w\big(D^2w:D^2(-e^{-w}) + e^{-w}|\na w|^4\big)dx
	+ \delta\int_\Omega e^{2w} |\na w|^2 dx \nonumber \\
	&\phantom{xx}{}+ \eps\int_\Omega |D^2\phi|^2 dx
	+ \delta\int_\Omega (|\na\phi|^2+\phi^2) dx
	+ \eps \int_\Omega (1+e^w)w(1-e^{-w})dx\nonumber\\
	&\phantom{xx}{}+ \delta\int_\Omega e^{-(N+1) w}w(e^{w}-1) dx \nonumber \\
	&=: I_1+\cdots+I_8.
	\label{aux1}
\end{align}

To estimate the first integral $I_1$, we use the entropy density \eqref{1.h}, 
formulated in terms of the variables $(\rho,E)$,
$$
  h(\rho,\theta) = \widetilde h(\rho,E) 
	= \rho\log\rho + \bigg(1+\frac32\rho\bigg)\log\frac{E}{1+\frac32\rho}.
$$
The function $\widetilde h$ in the variables $(\rho,E)$ is convex, 
since the determinant of its Hessian,
$$
  D^2\widetilde h(\rho,E) = \begin{pmatrix}
	\frac{1}{\rho} + \frac94(1+\frac32\rho)^{-1} & -\frac32\frac{1}{E} \\
	-\frac32\frac{1}{E} & (1+\frac32\rho)E^{-2}
	\end{pmatrix}
$$
equals $(1+\frac32\rho)/(\rho E^{2})$, which is positive. 
This implies that
$$
  \widetilde h(\rho_1,E_1)-\widetilde h(\rho_2,E_2) 
	\le D\widetilde h(\rho_1,E_1)\cdot
	\begin{pmatrix} \rho_1-\rho_2 \\ E_1-E_2 \end{pmatrix}
	= (\rho_1-\rho_2)\phi + (E_1-E_2)(-e^{-w})
$$
for any $(\rho_1,E_2)$, $(\rho_2,E_2)>0$, and consequently,
$$
  I_1 \ge \frac{\sigma}{\tau}\int_\Omega\big(\widetilde h(\rho,E)
	-\widetilde h(\rho^{k-1},E^{k-1})\big)dx
	+ \frac{\sigma}{\tau}\int_\Omega (E-E^{k-1})dx.
$$
The second integral $I_2$ is nonnegative since
\begin{align*}
  M_{11} & |\na\phi|^2 + 2M_{12}e^{-w}\na\phi\cdot\na w
	+ M_{22}e^{-2w}|\na w|^2 \\
  &= \rho e^w|\na\phi|^2 + 5\rho e^w\na\phi\cdot\na w
	+ \bigg(1+\frac{35}{4}\rho e^w\bigg)|\na w|^2 \\
	&= \rho e^w\bigg(\frac18|\na\phi|^2 
	+ \frac78\bigg|\na\phi+\frac{20}{7}\na w\bigg|^2 
	+ \frac{45}{28}|\na w|^2\bigg) + |\na w|^2 \\
	&\ge \frac18\rho e^w|\na\phi|^2 + \big(1+\rho e^w\big)|\na w|^2.
\end{align*}
The integrals $I_3$, $I_7$, and $I_8$ are estimated according to
\begin{align*}
  I_3 &= \frac12\big(|D^2 w|^2 + |D^2 w - \na w\otimes \na w|^2 + |\na w|^4\big)
	\ge \frac12\big(|D^2 w|^2 + |\na w|^4\big), \\
	I_7 &= 2\eps \int_\Omega w\sinh(w)dx \ge \eps \int_\Omega w^2 dx, \\
	I_8 &= \delta\int_\Omega e^{-(N+1)w}w(e^w-1)dx
	\ge \delta\int_\Omega e^{-(N+1)w}1_{\{w>-2\}}dx \\
	&= \delta\int_\Omega e^{-(N+1)w}dx - \delta\int_{\{w\ge -2\}}e^{-(N+1)w}dx \\
	&\ge \delta\int_\Omega e^{-(N+1)w}dx 
	- \delta e^{2(N+1)}\operatorname{meas}(\Omega).
\end{align*}
Therefore, we obtain from \eqref{aux1}
\begin{align}
  \frac{\sigma}{\tau}&\int_\Omega \big(\widetilde h(\rho,E)+E\big)dx
	+ \sigma\int_\Omega\bigg\{\frac18 \rho e^w|\na\phi|^2
	+ \big(1+\rho e^w\big)|\na w|^2\bigg\} dx \nonumber \\
	&\phantom{xx}{}+ \eps\int_\Omega |D^2\phi|^2 dx
	+ \frac{\eps}{2}\int_\Omega(|D^2 w|^2+|\na w|^4+w^2)dx
	 \nonumber \\
	&\phantom{xx}{}+ \delta\int_\Omega e^{2w}|\na w|^2 dx
	+ \delta\int_\Omega (|\na\phi|^2+\phi^2) dx 
	+ \delta\int_\Omega e^{-(N+1) w}dx\nonumber \\
  &\le \frac{\sigma}{\tau}\int_\Omega \big(\widetilde h(\rho^{k-1},E^{k-1})
	+E^{k-1}\big)dx + C\delta,
	\label{ei}
\end{align}
where $C>0$ is here and in the following a generic constant independent of
$\tau$, $\eps$, and $\delta$. 
This gives a uniform $H^2(\Omega)$ estimate for $\phi$ and $w$, independent of
$\sigma$ (but depending on $\eps$ and $\delta$), and hence the desired uniform estimate
for $(\phi,w)$ in $W^{1,4}(\Omega;\R^2)$. By the Leray--Schauder fixed-point theorem,
there exists a solution $(\phi^k,w^k):=(\phi,w)\in H^2(\Omega;\R^2)$ to 
\eqref{ap1}--\eqref{ap2} with $\sigma=1$, $\rho^k=\rho(\phi^k,w^k)$, and 
$E^k=E(\phi^k,w^k)$. Moreover, this solution satisfies \eqref{ei} with $\sigma=1$.

We reformulate equations \eqref{ap1}--\eqref{ap2} by inserting definition \eqref{2.M}
of the diffusion coefficients and computing (we drop the superindex $k$)
\begin{align*}
  M_{11}\na\phi + M_{12}e^{-w}\na w
	&= \rho\theta\na\bigg(\log\rho-\frac32\log\theta\bigg)
	+ \frac{5}{2}\rho\na\theta = \na(\rho\theta), \\
  M_{12}\na\phi + M_{22}e^{-w}\na w
  &= \frac{5}{2}\rho\theta^2\na\bigg(\log\rho-\frac32\log\theta\bigg) 
	+ \bigg(1+\frac{35}{4}\rho\theta\bigg)\na\theta
	= \na\bigg(\theta + \frac{5}{2}\rho\theta^2\bigg).
\end{align*}
Therefore, $(\phi^k,\rho^k,\theta^k,w^k)$ solves
\begin{align}
  0 &= \frac{1}{\tau}\int_\Omega(\rho^k-\rho^{k-1})\psi_1 dx
	+ \int_\Omega\na(\rho^k\theta^k)\cdot\na\psi_1 dx 
	+ \eps\int_\Omega D^2\phi^k:D^2\psi_1 dx \label{ap3} \\
	&\phantom{xx}{}+ \delta\int_\Omega (\na\phi^k\cdot\na\psi_1 + \phi^k\psi_1) dx, 
	\nonumber \\
  0 &= \frac{1}{\tau}\int_\Omega(E^k-E^{k-1})\psi_2 dx 
	+ \int_\Omega\na\bigg(\theta^k+\frac{5}{2}\rho^k(\theta^k)^2\bigg)\cdot\na\psi_2 dx
	\label{ap4} \\
	&\phantom{xx}{}+ \delta\int_\Omega e^{3w^k}\na w^k\cdot\na\psi_2 dx
	+ \eps \int_\Omega(1+e^{w^k})w^k\psi_2 dx \nonumber \\
  &\phantom{xx}{}+ \eps\int_\Omega e^{w^k}\big(D^2 w^k:D^2\psi_2 
	+ |\na w^k|^2\na w^k\cdot\na\psi_2\big)dx 
	+\delta\int_\Omega e^{-N w^k}w^k\psi_2 dx \nonumber
\end{align}
for test functions $\psi_1$, $\psi_2\in H^2(\Omega)$. 


\subsection{Uniform estimates}

Set $\theta^{k-1}=\exp(w^{k-1})$ and $\theta^k=\exp(w^k)$. In the
following, we drop again the superindex $k$ to simplify the notation. 
We reformulate inequality \eqref{ei} to obtain gradient estimates for 
expressions depending on $\rho$ and $\theta$.
We estimate the second integral in \eqref{ei}:
\begin{align*}
  \frac18 & \rho e^w|\na\phi|^2 + \big(1+\rho e^w\big)|\na w|^2 \\
	&= |\na\log\theta|^2 + \frac18\rho\theta\bigg|\frac{\na\rho}{\rho}
	-\frac32\frac{\na\theta}{\theta}\bigg|^2
	+ \rho\frac{|\na\theta|^2}{\theta} \\
	&= |\na\log\theta|^2 + \frac18\rho\theta\bigg(\frac{|\na\rho|^2}{\rho^2}
	- 3\frac{\na\rho}{\rho}\cdot\frac{\na\theta}{\theta}
	+ \frac{41}{4}\frac{|\na\theta|^2}{\theta^2} \bigg) \\
	&\ge |\na\log\theta|^2 + \frac{1}{16}\rho\theta\bigg(\frac{|\na\rho|^2}{\rho^2}
	+ \frac{|\na\theta|^2}{\theta^2}\bigg) \\
	&= |\na\log\theta|^2 + \frac{1}{32}\rho\theta\bigg(\frac{|\na\rho|^2}{\rho^2}
	+ \frac{|\na\theta|^2}{\theta^2}\bigg) \\
	&\phantom{xx}{}+ \frac{1}{32}\big|\sqrt{\theta}\na\sqrt{\rho}
	+ \sqrt{\rho}\na\sqrt{\theta}\big|^2
	+ \frac{1}{32}\big|\sqrt{\theta}\na\sqrt{\rho} - \sqrt{\rho}\na\sqrt{\theta}\big|^2 \\
	&\ge |\na\log\theta|^2 + \frac18\theta|\na\sqrt{\rho}|^2
	+ \frac{1}{32}\big|\na\sqrt{\rho\theta}\big|^2.
\end{align*}
We infer from \eqref{ei} with $\sigma=1$ the reformulated discrete entropy inequality
\begin{align}
  \frac{1}{\tau}\int_\Omega & \big(\widetilde h(\rho,E)+E\big)dx
	+ \int_\Omega|\na\log\theta|^2 dx + \frac18\int_\Omega\theta|\na\sqrt{\rho}|^2 dx
	+ \frac{1}{64}\int_\Omega|\na\sqrt{\rho\theta}|^2 dx \nonumber \\
  &\phantom{xx}{}+ \eps\int_\Omega |D^2\phi|^2 dx
	+ \frac{\eps}{2}\int_\Omega(|D^2 w|^2 + |\na w|^4 + w^2)dx 
	\nonumber \\
	&\phantom{xx}{}+ \delta\int_\Omega (|\na\phi|^2+ \phi^2) dx 
	+ \delta\int_\Omega |\na e^{w}|^2 dx + \delta\int_\Omega e^{-(N+1) w}dx	
	\nonumber\\
	&\le \frac{1}{\tau}\int_\Omega\big(\widetilde h(\rho^{k-1},E^{k-1})+E^{k-1}\big)dx 
	+ \delta C.	\label{ei2}
\end{align}
There exists $c\in(0,1)$ such that $x-\log x\ge c(x+|\log x|)$ for all $x>0$.
Therefore, 
\begin{align*}
  \widetilde h(\rho,E) + E 
	&= \rho\log\rho - \frac32\rho\log\theta - \log\theta 
	+ \bigg(1+\frac32\rho\bigg)\theta \nonumber \\
	&= \rho\log\rho + \bigg(1+\frac32\rho\bigg)(\theta-\log\theta)
  \ge \rho\log\rho + c(1+\rho)(\theta+|\log\theta|). 
\end{align*}
This provides the following uniform estimates independent of $(\delta,\eps,\tau)$:
\begin{equation}\label{aux22}
  \|\rho\log\rho\|_{L^1(\Omega)} + \|\theta\|_{L^1(\Omega)}
	+ \|\rho\theta\|_{L^1(\Omega)} + \|\log\theta\|_{L^1(\Omega)} \le C.
\end{equation}

\subsection{Limit $\eps\to 0$}

Let $\phi_\eps=\phi^k$, $w_\eps=w^k$ be a solution to \eqref{ap1}--\eqref{ap2}.
We set $\rho_\eps=\rho(\phi_\eps,w_\eps)$, $E_\eps=E(\rho_\eps,w_\eps)$,
$\theta_\eps=\exp(w_\eps)$, and $\phi_\eps=\log(\rho_\eps/\theta_\eps^{3/2})+5/2$. 
We deduce from \eqref{ei2} and \eqref{aux22} the following bounds which are
independent of $\eps$ and $\delta$ (but not of $\tau$):
\begin{align*}
  \|\rho_\eps\log\rho_\eps\|_{L^1(\Omega)}
	+ \|\theta_\eps\|_{L^1(\Omega)} + \|\rho_\eps\theta_\eps\|_{L^1(\Omega)}
	+ \|\log\theta_\eps\|_{L^1(\Omega)} &\le C, \\
  \|\sqrt{\theta_\eps}\na\sqrt{\rho_\eps}\|_{L^2(\Omega)}
	+ \|\na\sqrt{\rho_\eps\theta_\eps}\|_{L^2(\Omega)}
	+ \|\na\log\theta_\eps\|_{L^2(\Omega)} &\le C(\tau), \\
	\sqrt{\eps}\|\phi_\eps\|_{H^2(\Omega)} + \sqrt{\eps}\|w_\eps\|_{H^2(\Omega)}
	&\le C(\tau), \\
	\sqrt{\delta}\|\phi_\eps\|_{H^1(\Omega)} 
	+ \sqrt{\delta}\|\na\theta_\eps\|_{L^2(\Omega)}
	+ \delta\|\theta_\eps^{-(N+1)}\|_{L^1(\Omega)} &\le C(\tau).
\end{align*}
These bounds allow us to derive further estimates.  By the Poincar\'e inequality,
we have
\begin{align*}
  \|\theta_\eps\|_{L^2(\Omega)} &\le C\|\na\theta_\eps\|_{L^2(\Omega)}
	+ \|\theta_\eps\|_{L^1(\Omega)} \le C(\tau)\delta^{-1/2}, \\
	\|\log\theta_\eps\|_{L^2(\Omega)} &\le C\|\na\log\theta_\eps\|_{L^2(\Omega)}
	+ \|\log\theta_\eps\|_{L^1(\Omega)} \le C(\tau).
\end{align*}
This gives $\eps$-uniform bounds for $\theta_\eps$ and $\log\theta_\eps$
in $H^1(\Omega)$:
$$
  \|\theta_\eps\|_{H^1(\Omega)} \le C(\tau)\delta^{-1/2}, \quad
	\|\log\theta_\eps\|_{H^1(\Omega)} \le C(\tau).
$$
The $L^1(\Omega)$ bound for $\rho_\eps\theta_\eps$ and the $L^2(\Omega)$ bound for
$\na\sqrt{\rho_\eps\theta_\eps}$ imply that
$$
  \|\sqrt{\rho_\eps\theta_\eps}\|_{H^1(\Omega)} \le C(\tau).
$$
These estimates provide a uniform bound for the energy. Indeed, we deduce from 
the Sobolev embedding $H^1(\Omega)\hookrightarrow L^6(\Omega)$ that
$\na(\rho_\eps\theta_\eps)=2\sqrt{\rho_\eps\theta_\eps}\na\sqrt{\rho_\eps\theta_\eps}$
is uniformly bounded in $L^{3/2}(\Omega)$. This shows that $(E_\eps)$ is bounded
in $W^{1,3/2}(\Omega)$. 

We know that $(\log\theta_\eps)$ and $(\phi_\eps)$ are bounded in $H^1(\Omega)$.
Consequently, $\log\rho_\eps=\phi_\eps + \frac32\log\theta_\eps-\frac52$ is
bounded in $H^1(\Omega)$ too, i.e.
$$
  \sqrt{\delta}\|\log\rho_\eps\|_{H^1(\Omega)} \le C.
$$

The previous uniform bounds are sufficient to perform the limit $\eps\to 0$.
There exist subsequences which are not relabeled such that, as $\eps\to 0$,
\begin{align*}
  \phi_\eps\to \phi &\quad\mbox{strongly in }L^{p}(\Omega)\mbox{ and weakly in }
	H^1(\Omega), \\
  \log\rho_\eps \to Y &\quad\mbox{strongly in }L^{p}(\Omega)
	\mbox{ and weakly in } H^{1}(\Omega), \\
  \theta_\eps\to \theta &\quad\mbox{strongly in }L^{p}(\Omega)
	\mbox{ and weakly in }H^1(\Omega), \\
  \log\theta_\eps\to Z &\quad\mbox{strongly in }L^{p}(\Omega)
	\mbox{ and weakly in }H^1(\Omega), \\
  \eps\phi_\eps,\ \eps w_\eps \to 0 &\quad\mbox{strongly in }H^2(\Omega),\\
  \eps^{1/3}\na w_\eps\to 0 &\quad\mbox{strongly in }L^4(\Omega),
\end{align*}
where $1<p<6$ and $Y$, $Z$ are functions in $H^1(\Omega)$.
Up to a subsequence, we have $\log\rho_\eps\to Y$ and $\log\theta_\eps\to Z$ 
a.e.\ in $\Omega$. Thus, $\rho_\eps\to e^Y=:\rho$ and $\theta_\eps\to e^Z=:\theta$ 
a.e.\ in $\Omega$. In particular, $\rho>0$ and $\theta>0$ a.e.\ in $\Omega$.
It follows from
$$
  \int_{\{\rho_\eps\ge R\}}\rho_\eps dx
	\le \frac{1}{\log R}\int_{\{\rho_\eps\ge R\}}\rho_\eps\log\rho_\eps dx 
	\le \frac{C}{\log R}
$$
for any $R>1$
that $(\rho_\eps)$ is equi-integrable. Vitali's convergence theorem implies that
$\rho_\eps\to \rho$ strongly in $L^1(\Omega)$. 
Furthermore, possibly for a subsequence, $\sqrt{\rho_\eps\theta_\eps}\to
\sqrt{\rho\theta}$ a.e.\ in $\Omega$. The $H^1(\Omega)$ bound for 
$(\sqrt{\rho_\eps\theta_\eps})$ then yields
$$
  \sqrt{\rho_\eps\theta_\eps}\to\sqrt{\rho\theta}
	\quad\mbox{strongly in }L^{p}(\Omega)\mbox{ and weakly in }H^1(\Omega)\mbox{ and }
	L^6(\Omega),
$$
where $1<p<6$. Furthermore, we have
\begin{align*}
  E_\eps = \bigg(1+\frac32\rho_\eps\bigg)\theta_\eps
	\rightharpoonup E:=\bigg(1+\frac32\rho\bigg)\theta
	&\quad\mbox{weakly in }L^3(\Omega), \\
  \na(\rho_\eps\theta_\eps)=2\sqrt{\rho_\eps\theta_\eps}\na\sqrt{\rho_\eps\theta_\eps}
	\rightharpoonup\na(\rho\theta) &\quad\mbox{weakly in }L^{3/2}(\Omega), \\
	\na(\rho_\eps\theta_\eps^2) = \rho_\eps\theta_\eps\na\theta_\eps
	+ \theta_\eps\na(\rho_\eps\theta_\eps) \rightharpoonup \na(\rho\theta^2)
	&\quad\mbox{weakly in }L^{6/5}(\Omega).
\end{align*}
We deduce from the strong convergence of $(\phi_\eps)$, $(\rho_\eps)$, and 
$(\theta_\eps)$ as well as from the a.e.\ positivity of $\rho$ and $\theta$ that
$\phi=\log\rho-\frac32\log\theta+\frac52$ a.e.\ in $\Omega$. 

The uniform bounds for $w_\eps$ are sufficient to pass to the limit $\eps\to 0$
in the $\eps$-terms,
\begin{align*}
  \eps D^2\phi_\eps \to 0 &\quad\mbox{strongly in }L^2(\Omega), \\
  \eps\theta_\eps D^2 w_\eps \to 0 &\quad\mbox{strongly in }L^1(\Omega), \\
  \eps\theta_\eps|\na w_\eps|^2\na w_\eps \to 0 &\quad\mbox{strongly in }L^1(\Omega), \\
	\eps(1+\theta_\eps)w_\eps \to 0 &\quad\mbox{strongly in }L^2(\Omega),
\end{align*}
as well as in the $\delta$-terms. The most difficult term is
$\delta\int_\Omega e^{-N w_\eps}w_\eps\psi_2 dx$. 
It follows from $\sqrt{\theta}|\log\theta|^{(2N+1)/(2N)}\le C$ for $\theta\le 1$
and $\theta^{-(N+1)}\sqrt{\theta}|\log\theta|^{(2N+1)/(2N)}\le C$ for
$\theta>1$ as well as from \eqref{aux22} that
\begin{align}
  \delta\| e^{-Nw_\eps}w_\eps\|_{L^{(2N+1)/(2N)}(\Omega)}^{(2N+1)/(2N)}
	&= \delta\int_\Omega \theta_\eps^{-(N+1)}\sqrt{\theta_\eps}
	|\log\theta_\eps|^{(2N+1)/(2N)}dx \nonumber \\
	&\le \delta\int_\Omega\theta_\eps^{-(N+1)}dx + \delta C\le C(\tau). \label{eww}
\end{align}
Since $\delta e^{-Nw_\eps}w_\eps\to \delta\theta^{-N}\log\theta$ a.e.\ in $\Omega$, 
we conclude that this limit also holds strongly in $L^1(\Omega)$. 
Therefore, we can perform the limit $\eps\to 0$ in \eqref{ap3}--\eqref{ap4}
(now writing the superindex $k$) leading to
\begin{align}
  0 &= \frac{1}{\tau}\int_\Omega(\rho^k-\rho^{k-1})\psi_1 dx
	+ \int_\Omega\na(\rho^k\theta^k)\cdot\na\psi_1 dx
	+ \delta\int_\Omega(\na\phi^k\cdot\na\psi_1 + \phi^k\psi_1)dx, \label{aux3} \\
  0 &= \frac{1}{\tau}\int_\Omega(E^k-E^{k-1})\psi_2 dx
	+ \int_\Omega\na\bigg(\theta^k + \frac{5}{2}\rho^k(\theta^k)^2\bigg)\cdot
	\na\psi_2 dx \label{aux4} \\
	&\phantom{xx}{}+ \delta\int_\Omega(\theta^k)^2\na\theta^k\cdot\na\psi_2 dx
	+ \delta\int_\Omega(\theta^k)^{-N}\log(\theta^k)\psi_2 dx \nonumber
\end{align}
for any test functions $\psi_1\in W^{1,3}(\Omega)$, $\psi_2\in W^{1,6}(\Omega)$.


\subsection{Limit $\tau\to 0$}

We introduce the piecewise constant functions in time $\rho_\tau(x,t)
=\rho^k(x)$, $\theta_\tau(x,t)=\theta^k(x)$, $\phi_\tau(x,t)=\phi^k(x)$,
and $E_\tau(x,t)=E^k(x)$ for
$x\in\Omega$, $t\in((k-1)\tau,k\tau]$. Furthermore, let $(\pi_\tau u)(x,t)
=u^{k-1}(x)$ for $x\in\Omega$, $t\in((k-1)\tau,k\tau]$ 
be the shift operator for
piecewise constant functions $u$. We reformulate \eqref{aux3}--\eqref{aux4}:
\begin{align}
  0 &= \frac{1}{\tau}\int_0^T\int_\Omega(\rho_\tau-\pi_\tau\rho_\tau)\psi_1 dx dt 
	+  \int_0^T\int_\Omega\na(\rho_\tau\theta_\tau)\cdot\na\psi_1 dx dt 
	\label{ap5} \\
  &\phantom{xx}{}+ \delta\int_0^T\int_\Omega (\na\phi_\tau\cdot\na\psi_1 
	+ \phi_\tau\psi_1) dx dt, \nonumber \\
  0 &= \frac{1}{\tau}\int_0^T\int_\Omega (E_\tau-\pi_\tau E_\tau)\psi_2 dx dt 
	+ \int_0^T\int_\Omega \na\bigg(\theta_\tau+\frac{5}{2}\rho_\tau\theta_\tau^2\bigg) 
	\cdot\na\psi_2 dx dt \label{ap6} \\
  &\phantom{xx}{}+ \delta\int_0^T\int_\Omega \theta_\tau^2\na\theta_\tau\cdot
	\na\psi_2 dx dt
  + \delta\int_0^T\int_\Omega \theta_\tau^{-N}\log(\theta_\tau)\psi_2 dx dt \nonumber
\end{align}
for piecewise constant test functions in time
$\psi_1$, $\psi_2\in L^2(0,T;W^{1,6}(\Omega))$. By density \cite[Prop.~1.36]{Rou05},
these formulations hold for all test functions in $L^2(0,T;W^{1,6}(\Omega))$.
We collect the uniform estimates
from the discrete entropy inequality \eqref{ei2}:
\begin{align}
  \label{bound.1}
  \|\rho_\tau\log\rho_\tau\|_{L^\infty(0,T;L^{1}(\Omega))}
	+ \|\theta_\tau\|_{L^\infty(0,T;L^1(\Omega))} &\le C, \\
  \label{bound.2}
  \|\rho_\tau\theta_\tau\|_{L^\infty(0,T;L^1(\Omega))}
	+ \|\sqrt{\theta_\tau}\na\sqrt{\rho_\tau}\|_{L^2(\Omega_T)}
	+ \|\sqrt{\rho_\tau\theta_\tau}\|_{L^2(0,T;H^1(\Omega))} &\le C, \\
  \label{bound.3}
  \|\log\theta_\tau\|_{L^2(0,T; H^1(\Omega))} 
	+ \|\log\theta_\tau\|_{L^\infty(0,T;L^1(\Omega))} & \le C,\\
  \label{bound.4}
  \sqrt{\delta}\|\phi_\tau\|_{L^2(0,T;H^1(\Omega))}
  + \sqrt{\delta}\|\na\theta_\tau\|_{L^2(\Omega_T)} 
  + \delta\|\theta_\tau^{-(N+1)}\|_{L^1(\Omega_T)}	 &\le C,
\end{align}
where the constant $C>0$ does not depend on $\tau$ or $\delta$. 
In the following, we show some additional estimates for $(\rho_\tau,\theta_\tau)$.

\begin{lemma}[Mass and energy control]\label{lem.mass}
It holds for any $t\in(0,T)$ that
$$
  \bigg|\int_\Omega\rho_\tau(t)dx - \int_\Omega\rho^0 dx\bigg|\le C\delta^{1/2}, \quad
	\bigg|\int_\Omega E_\tau(t)dx - \int_\Omega E^0 dx\bigg|
	\le C\delta^{1/(2N+1)}.
$$
\end{lemma}

\begin{proof}
Using $\psi_1=1$ in \eqref{aux3} and summing from $k=1,\ldots,n$ gives
$$
  \int_\Omega(\rho^n-\rho^0)dx
	= \sum_{k=1}^n\int_\Omega(\rho_\tau ^k-\rho^{k-1})dx 
	= \tau\delta\sum_{k=1}^n\int_\Omega\phi^k dx,
$$
where $n\le N$. We infer from bound \eqref{bound.4} for $(\phi_\tau)$ that
$$
  \bigg|\int_\Omega(\rho_\tau(t)-\rho^0)dx\bigg| 
	= \delta\bigg|\int_0^t\int_\Omega\phi_\tau dxdt\bigg|
	\le \delta^{1/2}C.
$$
The second statement follows after choosing $\psi_2=1$ in \eqref{aux4} and
using \eqref{eww}. 
\end{proof}

\begin{lemma}[Higher integrability]\label{lem.higher}
It holds that
\begin{align*}
  \|\theta_\tau\|_{L^2(\Omega_T)} 
	+ \|\rho_\tau^\alpha\theta_\tau^\beta\|_{L^1(\Omega_T)} 
	+ \delta^{1/4}\|\theta_\tau\|_{L^4(\Omega_T)} &\le C, 
\end{align*}
where $(\alpha,\beta)\in\{(1,2),(1,3),(\frac32,3),(2,1),(2,2),(2,3)\}$. 
\end{lemma}

\begin{proof}
The proof is based on the $H^{-1}(\Omega)$ method, i.e., we use test functions
of the type $(-\Delta)^{-1}\rho_\tau$ and $(-\Delta)^{-1}E_\tau$. More precisely,
let $\Psi_1$, $\Psi_2\in L^\infty(0,T;H^1(\Omega))$ be the unique solutions to,
respectively,
\begin{equation}\label{Psi}
\begin{aligned}
  -\Delta\Psi_1 &= \rho_\tau - \fint_\Omega\rho_\tau dx\quad\mbox{on }\Omega,
	\quad\na\Psi_1\cdot\nu=0\quad\mbox{on }\pa\Omega ,\quad \int_\Omega\Psi_1 dx = 0,\\
  -\Delta\Psi_2 &= E_\tau - \fint_\Omega E_\tau dx
  \quad\mbox{on }\Omega,\quad \na\Psi_2\cdot\nu=0\quad\mbox{on }\pa\Omega ,
	\quad\int_\Omega\Psi_2 dx = 0,
\end{aligned}
\end{equation}
where $\fint udx = \operatorname{meas}(\Omega)^{-1}\int_\Omega udx$.

{\em Step 1: uniform bounds for $\Psi_2$.} We use the test function $\Psi_2$
in the weak formulation of the second equation in \eqref{Psi} and take into
account the energy control. Then
$$
  \|\na\Psi_2\|_{L^2(\Omega)}^2 \le C\big(1+\|E_\tau\|_{L^{6/5}(\Omega)}\big)
	\|\Psi_2\|_{L^6(\Omega)}.
$$
It follows from Sobolev's embedding and the Poincar\'e--Wirtinger inequality
that
$$
  \|\na\Psi_2\|_{L^2(\Omega)}^2 \le C\big(1+\|E_\tau\|_{L^{6/5}(\Omega)}\big)
	\|\na\Psi_2\|_{L^2(\Omega)}
$$
and so
$$
  \|\Psi_2\|_{H^1(\Omega)}\le  C\big(1+\|E_\tau\|_{L^{6/5}(\Omega)}\big).
$$

We proceed by bootstrapping this result. Elliptic regularity for
$$
  -\Delta\Psi_2 + \Psi_2 = E_\tau - \fint_\Omega E_\tau dx + \Psi_2
	\quad\mbox{in }\Omega
$$
gives (here, we need the boundary regularity $\pa\Omega\in C^{1,1}$)
$$
  \|\Psi_2\|_{W^{2,6/5}(\Omega)}
	\le C\big(1+\|E_\tau\|_{L^{6/5}(\Omega)} + \|\Psi_2\|_{L^{6/5}(\Omega)}\big)
	\le C\big(1+\|E_\tau\|_{L^{6/5}(\Omega)}\big).
$$
Since $(E_\tau)$ is bounded in $L^{\infty}(0,T;L^1(\Omega))$, 
an interpolation shows that
\begin{align*}
  \|E_\tau\|_{L^6(0,T;L^{6/5}(\Omega))}^6
	&\le \int_0^T\|E_\tau\|_{L^2(\Omega)}^2\|E_\tau\|_{L^1(\Omega)}^4 dt \\
	&\le \|E_\tau\|_{L^\infty(0,T;L^1(\Omega))}^4\int_0^T\|E_\tau\|_{L^2(\Omega)}^2 dt.
\end{align*}
We deduce from the embedding $L^6(0,T;W^{2,6/5}(\Omega))\hookrightarrow
L^6(\Omega_T)$ that
\begin{equation}\label{Psi2}
  \|\Psi_2\|_{L^6(\Omega_T)} \le C\big(1+\|E_\tau\|_{L^2(\Omega_T)}^{1/3}\big).
\end{equation}

{\em Step 2: Test functions $\Psi_1$ and $\Psi_2$.}
We choose $\Psi_1$ and $\Psi_2$ as test functions in \eqref{ap5} and \eqref{ap6},
respectively:
\begin{align}
  0 &= \frac{1}{\tau}\int_0^T\int_\Omega(\rho_\tau - \pi_\tau\rho_\tau)\Psi_1 dx dt
  + \int_0^T\int_\Omega \rho_\tau\theta_\tau\bigg(\rho_\tau
	- \fint_\Omega\rho_\tau dx\bigg)dxdt \label{ap5.1} \\
  &\phantom{xx}{} + \delta\int_0^T\int_\Omega \phi_\tau\bigg(\rho_\tau
	- \fint_\Omega\rho_\tau dx + \Psi_1\bigg) dxdt, \nonumber \\
  0 &= \frac{1}{\tau}\int_0^T\int_\Omega(E_\tau - \pi_\tau E_\tau)\Psi_2 dx dt
  + \int_0^T\int_\Omega\bigg(\theta_\tau+\frac{5}{2}\rho_\tau\theta_\tau^2\bigg)
  \bigg(E_\tau - \fint_\Omega E_\tau dx \bigg)dxdt \label{ap6.1} \\
  &\phantom{xx}{}+ \frac{\delta}{3}\int_0^T\int_\Omega\theta_\tau^3 
	\bigg(E_\tau - \fint_\Omega E_\tau dx \bigg) dxdt
  + \delta\int_0^T\int_\Omega \theta_\tau^{-N}\log(\theta_\tau)\Psi_2 dx dt.
  \nonumber
\end{align}

We estimate the first integral in \eqref{ap5.1}. Since $\Psi_1$ has zero 
spatial average and $\na\Psi_1\cdot\nu=0$ on $\pa\Omega$, it follows from
\eqref{Psi} that
\begin{align*}
  \frac{1}{\tau}\int_0^T\int_\Omega(\rho_\tau -\pi_\tau\rho_\tau)\Psi_1 dx dt
	&= \frac{1}{\tau}\int_0^T\int_\Omega(\mbox{id} - \pi_\tau)
	\bigg(\rho_\tau-\fint_\Omega\rho_\tau dx\bigg)\Psi_1 dx dt \\
  &\phantom{xx}{}+ \frac{1}{\tau}\int_0^T(\mbox{id} - \pi_\tau)
	\bigg(\fint_\Omega\rho_\tau dx\bigg)\bigg(\int_\Omega\Psi_1 dx\bigg) dt \nonumber \\
  &= \frac{1}{\tau}\int_0^T\int_\Omega \nabla\big((\mbox{id} - \pi_\tau)\Psi_1\big)
	\cdot \nabla\Psi_1 dx dt. 
\end{align*}
The function $\Psi_1$ is piecewise constant in time. We write
$\Psi_1(x,t)=\Psi_1^k(x)$ for $x\in\Omega$, $t\in((k-1)\tau,k\tau]$. Then, using
Young's inequality,
\begin{align*}
  \frac{1}{\tau}\int_0^T&\int_\Omega \nabla\big((\mbox{id} - \pi_\tau)\Psi_1\big)
	\cdot \nabla\Psi_1 dx dt
	= \sum_{k=1}^N\int_\Omega\na(\Psi_1^k-\Psi_1^{k-1})\cdot\na\Psi_1^k dx \\
	&\ge \frac12\sum_{k=1}^N\int_\Omega\big(|\na\Psi_1^k|^2-|\na\Psi_1^{k-1}|^2\big)dx
  = \frac12\int_\Omega\big(|\na\Psi_1^N|^2-|\na\Psi_1^0|^2\big)dx.
\end{align*}
We conclude that
$$
  \frac{1}{\tau}\int_0^T\int_\Omega(\rho_\tau -\pi_\tau\rho_\tau)\Psi_1 dx dt
	\ge \frac{1}{2}\int_\Omega |\nabla\Psi_1(T)|^2 dx
  - \frac{1}{2}\int_\Omega |\nabla\Psi_1(0)|^2 dx.
$$
In a similar way, we have
$$
  \frac{1}{\tau}\int_0^T\int_\Omega(E_\tau -\pi_\tau E_\tau)\Psi_2 dx dt
  \ge \frac{1}{2}\int_\Omega |\nabla\Psi_2(T)|^2 dx
  - \frac{1}{2}\int_\Omega |\nabla\Psi_2(0)|^2 dx.
$$
Inserting these inequalities into \eqref{ap5.1} and \eqref{ap6.1}, respectively,
and adding both inequalities, we find that
\begin{align}
  \frac{1}{2} & \int_\Omega|\nabla\Psi_1(T)|^2 dx 
	+ \frac{1}{2}\int_\Omega |\nabla\Psi_2(T)|^2 dx
  + \frac{\delta}{3}\int_0^T\int_\Omega\theta_\tau^3 E_\tau dxdt \nonumber \\
  &\phantom{xx}{}+ \int_0^T\int_\Omega \rho_\tau^2\theta_\tau dxdt
  + \int_0^T\int_\Omega\bigg(\theta_\tau+\frac{5}{2}\rho_\tau\theta_\tau^2\bigg)
  \bigg(\theta_\tau+\frac{3}{2}\rho_\tau\theta_\tau\bigg) dxdt \nonumber \\
  &\leq \frac{1}{2}\int_\Omega |\nabla\Psi_1(0)|^2 dx
  + \frac{1}{2}\int_\Omega |\nabla\Psi_2(0)|^2 dx 
  - \delta\int_0^T\int_\Omega\phi_\tau\bigg(\rho_\tau 
	- \fint_\Omega\rho_\tau dx +\Psi_1\bigg) dxdt \nonumber \\
  &\phantom{xx}{}+ \frac32\int_0^T\bigg(\int_\Omega \rho_\tau\theta_\tau dx\bigg) 
	\bigg(\fint_\Omega\rho_\tau dx \bigg) dt
  + \int_0^T \int_\Omega\bigg(\theta_\tau+\frac{5}{2}\rho_\tau\theta_\tau^2
	\bigg)dx\bigg(\fint_\Omega E_\tau dx \bigg) dt \nonumber \\ 
  &\phantom{xx}{}+ \frac{\delta}{3}\int_0^T\int_\Omega\theta_\tau^3 
	\bigg(\fint_\Omega E_\tau dx \bigg) dxdt 
  - \delta\int_0^T\int_\Omega \theta_\tau^{-N}\log(\theta_\tau)\Psi_2 dx dt \nonumber \\
	&=: J_1+\cdots+J_7. \label{J17}
\end{align}

We start with the last integral. It follows from \eqref{Psi2} that
$$
  J_7 \le \delta \|\theta_\tau^{-N}\log\theta_\tau\|_{L^{6/5}(\Omega_T)}
	\|\Psi_2\|_{L^6(\Omega_T)}
  \le \delta C\|\theta_\tau^{-N}\log\theta_\tau\|_{L^{6/5}(\Omega_T)}
	\big(1+\|E_\tau\|_{L^2(\Omega_T)}^{1/3}\big).
$$
The first norm is estimated according to	
\begin{align*}
  \|\theta_\tau^{-N}\log\theta_\tau\|_{L^{6/5}(\Omega_T)}^{6/5}
	&= \int_0^T\int_\Omega\theta_\tau^{-6N/5}|\log\theta_\tau|^{6/5}dxdt \\
	&\le C + \int_0^T\int_{\Omega\cap\{\theta_\tau(t)<1\}}
	\theta_\tau^{-6N/5}|\log\theta_\tau|^{6/5}dxdt \\
  &\le C + C\int_0^T\int_{\Omega\cap\{\theta_\tau(t)<1\}}
	\theta_\tau^{-(N+1)}dxdt,
\end{align*}
where the last inequality follows from the condition $N<5$ (and hence
$6N/5<N+1$). Because of \eqref{bound.4}, this leads to
\begin{equation}\label{est.L65}
  \delta\|\theta_\tau^{-N}\log\theta_\tau\|_{L^{6/5}(\Omega_T)} 
	\le C\delta^{1/6}.
\end{equation}
Therefore, we infer that
$$
  J_7 \le \delta^{1/6}C\big(1+\|E_\tau\|_{L^2(\Omega_T)}^{1/3}\big).
$$
Since $E_\tau=\theta_\tau+\frac32\rho_\tau\theta_\tau$, the right-hand side
can be controlled (for sufficiently small $\delta>0$) by the last two integrals
on the left-hand side of \eqref{J17}.

Next, we consider the following term appearing in $J_3$:
\begin{align*}
  -\delta\int_0^T\int_\Omega\phi\rho_\tau dx dt 
	&= -\delta\int_0^T\int_\Omega\bigg(\log(\rho_\tau\theta_\tau^{-3/2}) 
	+ \frac52\bigg)\rho_\tau dx dt\\
  &\le -\delta\int_0^T\int_\Omega\theta_\tau^{3/2}\cdot\rho_\tau\theta_\tau^{-3/2}
	\log(\rho_\tau\theta_\tau^{-3/2}) dx dt + C \\
  &\le \delta C\int_0^T\int_\Omega\theta_\tau^{3/2} dx dt + C,
\end{align*}
where the last inequality follows from the fact that $z\mapsto z\log z$ is bounded
from below. Furthermore, we deduce from Lemma \ref{lem.mass},
bound \eqref{bound.4} for $\phi_\tau$, and the Poincar\'e--Wirtinger
inequality that
\begin{align*}
  \delta\int_0^T\int_\Omega\phi_\tau\bigg(\fint_\Omega\rho_\tau dx\bigg)dxdt
	&\le \delta\|\phi_\tau\|_{L^1(\Omega_T)}\|\rho_\tau\|_{L^\infty(0,T;L^1(\Omega))}
	\le C, \\
  \delta\int_0^T\int_\Omega\phi_\tau\Psi_1 dxdt
	&\le \frac{\delta}{2}\int_0^T\int_\Omega\phi_\tau^2 dxdt
	+ \frac{\delta}{2}\int_\Omega\Psi_1^2 dxdt \\
	&\le C + \delta C\int_0^T\int_\Omega|\na\Psi_1|^2 dxdt. 
\end{align*}
This shows that 
$$
  J_3 \le C + \delta C \int_0^T\int_\Omega\theta_\tau^{3/2} dx dt 
	+ \delta C\int_0^T\int_\Omega|\na\Psi_1|^2 dxdt.
$$
The first integral on the right-hand side can be controlled by the last integral on
the left-hand side of \eqref{J17}. The last integral on the right-hand side
is controlled after applying Gronwall's inequality.
The integrals $J_4$, $J_5$, and $J_6$ can
be controlled by the expressions on the left-hand side of \eqref{J17}. 
We conclude that
\begin{align*}
  \int_\Omega & \big(|\na\Psi_1(T)|^2 + |\na\Psi_2(T)|^2\big)dx \\
	&{}+ \int_0^T\int_\Omega\big(\theta_\tau^2 + \delta\theta_\tau^4
	+ \rho_\tau\theta_\tau^2(1+\theta_\tau) + \rho_\tau^2\theta_\tau(1+\theta_\tau^2)
	\big)dxdt \le C\exp(\delta CT).
\end{align*}
We deduce from this estimate and Young's inequality that
\begin{align*}
  \|\rho_\tau\theta_\tau\|_{L^2(\Omega_T)}^2
	&\le \frac12\int_0^T\int_\Omega\rho_\tau^2(\theta_\tau+\theta_\tau^3) dxdt \le C, \\
	\|\rho_\tau\theta_\tau^2\|_{L^{3/2}(\Omega_T)}^{3/2}
	&\le \frac12\int_0^T\int_\Omega(\rho_\tau+\rho_\theta^2)\theta_\tau^3 dxdt \le C.
\end{align*}
This proves the lemma.
\end{proof}

{\em Step 3: Strong convergence of $(\rho_\tau)$ and $(\theta_\tau)$.}
First, we prove a gradient bound for the particle density.

\begin{lemma}[Gradient estimate]\label{lem.grad}
There exist $N\in(0,5)$, $m\in(\frac12,1)$, and $\alpha\in(\frac23,1)$ such that
$$
  \|\rho_\tau^m\|_{L^p(0,T;W^{1,q}(\Omega))} \le C(\delta),
$$
where $C(\delta)>0$ does not depend on $\tau$, $p\ge 1/m$, and
$3q/(3-q)>1/m$ (or equivalently, $q>3/(3m+1)$). Moreover, with a constant
$C>0$ independent of $\tau$ and $\delta$,
$$
  \|E_\tau\|_{L^1(0,T;W^{1,1}(\Omega))} \le C.
$$
\end{lemma}

The condition $q>3/(3m+1)$ guarantees that $W^{1,q}(\Omega)\hookrightarrow
L^{1/m}(\Omega)$. This is needed below for the application of the 
nonlinear Aubin--Lions lemma.

\begin{proof}
It follows from Lemma \ref{lem.higher} that $(\rho_\tau\theta_\tau^{1/2})$
is bounded in $L^2(\Omega_T)$, while estimate \eqref{bound.4} implies that
$(\theta_\tau^{-1/2})$ is bounded in $L^{2(N+1)}(\Omega_T)$. Consequently,
$\rho_\tau = \rho_\tau\theta^{1/2}\theta_\tau^{-1/2}$ is uniformly bounded in
$L^{r}(\Omega_T)$, where $r:=2(N+1)/(N+2)>1$. 
Together with the $L^\infty(0,T;L^1(\Omega))$
bound for $(\rho_\tau)$, an interpolation with $1/c=(1-\alpha)/1 + \alpha/r$ 
and $b\ge 1$ gives
$$
  \|\rho_\tau\|_{L^b(0,T;L^c(\Omega))}^b
	\le \|\rho_\tau\|_{L^\infty(0,T;L^1(\Omega))}^{(1-\alpha)b}
	\int_0^T\|\rho_\tau\|_{L^r(\Omega)}^{\alpha b}dt 
	\le C\int_0^T\|\rho_\tau\|_{L^r(\Omega)}^{\alpha b}dt.
$$
A simple computation shows that $c=r/(\alpha+(1-\alpha)r)$. 
We choose $b=r/\alpha$ and use the $L^r(\Omega_T)$ bound for $(\rho_\tau)$:
$$
  \|\rho_\tau\|_{L^{r/\alpha}(0,T;L^{r/(\alpha+(1-\alpha)r)}(\Omega))} \le C
	\quad\mbox{for } r=\frac{2(N+1)}{N+2},\ \alpha\in(0,1).
$$
Let $\frac12<m<1$. Then
$$
  \|\rho_\tau^m\|_{L^{r/(\alpha m)}(0,T;L^{r/(m(\alpha+(1-\alpha)r))}(\Omega))}\le C.
$$
We know from \eqref{bound.3} and \eqref{bound.4} that
$\na\log\rho_\tau = \na\phi_\tau + \frac32\na\log\theta_\tau$ is uniformly bounded in
$L^2(\Omega_T)$ (but not uniformly in $\delta$). It follows that
$\na\rho_\tau^m = m\rho_\tau^m\na\log\rho_\tau$ is uniformly bounded in
$L^p(0,T;L^q(\Omega))$, where $p,q\ge 1$ satisfy 
\begin{equation}\label{pq}
  \frac{1}{p} = \frac12 + \frac{\alpha m}{r}, \quad
	\frac{1}{q} = \frac12 + \frac{m}{r}(\alpha+(1-\alpha)r).
\end{equation} 
We deduce from the Poincar\'e--Wirtinger inequality and the 
$L^\infty(0,T;L^1(\Omega))$ bound for $(\rho_\tau)$ that
$$
  \|\rho_\tau^m\|_{L^p(0,T;L^q(\Omega))} 
	\le C\|\na\rho_\tau^m\|_{L^p(0,T;L^q(\Omega))} 
	+ C\|\rho_\tau^m\|_{L^p(0,T;L^1(\Omega))} \le C(\delta).
$$

We claim that there exist $N\in(0,5)$, $m\in(\frac12,1)$, and $\alpha\in(0,1)$
such that 
$$
  p\ge\frac{1}{m}, \quad \frac{3q}{3-q} > \frac{1}{m},
$$
where $p$ and $q$ are given by \eqref{pq}.
A straightforward computation shows that these inequalities are equivalent to
$$
  r\ge \frac{2\alpha m}{2m-1}, \quad \frac{r}{r-1} < 6\alpha m.
$$
We choose $r=2\alpha m/(2m-1)$ (recall that $m>1/2$) such that the first inequality 
is satisfied. With this choice, the second inequality is equivalent to 
$m<1/(3(1-\alpha))$. Since we want $m<1$, we need to choose 
$\alpha>2/3$. Then $\frac12<m<1<1/(3(1-\alpha))$. By definition of $r$,
\begin{equation}\label{r}
  \frac{2(N+1)}{N+2} = r = \frac{2\alpha m}{2m-1}.
\end{equation}
Thus, it remains to prove that $N\in(0,5)$ can be chosen such that
this identity holds for some $\alpha>\frac23$ and $m\in(\frac12,1)$. 
Equation \eqref{r} is equivalent to
$$
  N = -\frac{2\alpha m-2m+1}{\alpha m-2m+1},
$$
and the requirement $N<5$ gives $m>6/(12-7\alpha)$. The right-hand side is
smaller than one if $\alpha<\frac67$. This is compatible with the previous constraint
$\alpha>\frac23$ and proves the claim.

To finish the proof of the lemma, we observe that
\eqref{bound.2} and Lemma \ref{lem.higher} imply that
\begin{equation}\label{rhoth}
  \|\na(\rho_\tau\theta_\tau)\|_{L^{4/3}(\Omega_T)}
	\le 2\|\sqrt{\rho_\tau\theta_\tau}\|_{L^4(\Omega_T)}
	\|\na\sqrt{\rho_\tau\theta_\tau}\|_{L^2(\Omega_T)} \le C.
\end{equation}
Moreover, we deduce from \eqref{bound.3} and Lemma \ref{lem.higher} that
$$
  \|\na\theta_\tau\|_{L^1(\Omega_T)}
	\le \|\theta_\tau\|_{L^2(\Omega_T)}\|\na\log\theta_\tau\|_{L^2(\Omega_T)} \le C.
$$
Thus, $(E_\tau)$ is bounded in $L^1(0,T;W^{1,1}(\Omega))$, and the proof is finished.
\end{proof}

\begin{lemma}[Bounds for the discrete time derivative]\label{lem.time}
There exists a constant $C>0$ which does not depend on $\tau$ such that
$$
  \tau^{-1}\|\rho_\tau-\pi_\tau\rho_\tau\|_{L^{4/3}(0,T;W^{1,4}(\Omega)')} \le C,
	\quad \tau^{-1}\|E_\tau-\pi_\tau E_\tau\|_{L^{6/5}(0,T;W^{2,4}(\Omega)')} \le C.
$$
\end{lemma}

\begin{proof}
We infer from \eqref{rhoth} and \eqref{bound.4} that
\begin{align*}
  \tau^{-1}\|\rho_\tau-\pi_\tau\rho_\tau\|_{L^{4/3}(0,T;W^{1,4}(\Omega)')}
	&= \sup_{\|\psi_1\|_{L^4(0,T;W^{1,4}(\Omega))}=1}
	\bigg|\tau^{-1}\int_0^T\int_\Omega(\rho_\tau-\pi_\tau\rho_\tau)\psi_1 dxdt\bigg| \\
	&\le \frac32\|\na(\rho_\tau\theta_\tau)\|_{L^{4/3}(\Omega_T)}
	+ \delta\|\phi_\tau\|_{L^{4/3}(\Omega_T)} \le C.
\end{align*}
Furthermore,
\begin{align*}
  \tau^{-1}\|E_\tau-\pi_\tau E_\tau\|_{L^{6/5}(0,T;W^{2,4}(\Omega)')}
	&\le \|\theta_\tau\|_{L^2(\Omega_T)}
	+ \frac{15}{4}\|\rho_\tau\theta_\tau^2\|_{L^{3/2}(\Omega_T)} \\
	&\phantom{xx}{}+ \frac{\delta}{3}\|\theta_\tau^3\|_{L^{4/3}(\Omega_T)}
	+ \delta\|\theta_\tau^{-N}\log\theta_\tau\|_{L^{6/5}(\Omega_T)}.
\end{align*}
Taking into account Lemma \ref{lem.higher}, the first three terms on the 
right-hand side are uniformly bounded. Since $N<5$, the last term can be estimated
from above by $\delta\|\theta_\tau^{-(N+1)}\|_{L^1(\Omega_T)}^{5/6}$ which is bounded
because of \eqref{bound.4}. This finishes the proof.
\end{proof}

Lemmas \ref{lem.grad} and \ref{lem.time} allow us to apply the
Aubin--Lions lemma in the version of \cite[Theorem 3]{CJL14}.
This is possible since $p\ge 1/m$ and $W^{1,q}(\Omega)\hookrightarrow
L^{1/m}(\Omega)$ (the last fact is a consequence of $q>3/(3m+1)$). 
We infer the existence of a subsequence which is not relabeled such that,
as $\tau\to 0$,
$$
  \rho_\tau\to \rho\quad\mbox{strongly in }L^1(\Omega_T).
$$
Concerning $(E_\tau)$, Lemmas \ref{lem.grad} and \ref{lem.time} allow us 
to apply the Aubin--Lions lemma in the version of \cite{DrJu12}
(or Theorem 3 in \cite{CJL14} with $m=1$) to obtain a subsequence of
$(E_\tau)$ (not relabeled) such that, as $\tau\to 0$,
$$
  E_\tau\to E\quad\mbox{strongly in }L^1(\Omega_T).
$$
In fact, because of the $L^2(\Omega_T)$ bound for $(E_\tau)$ from Lemma
\ref{lem.higher}, this convergence holds in $L^\eta(\Omega_T)$ for any $\eta<2$.
Up to subsequences, we know that $\rho_\tau\to\rho$ and $E_\tau\to E$ a.e.\
in $\Omega_T$. Thus,
$$
  \theta_\tau = \frac{E_\tau}{1+3\rho_\tau/2} \to \frac{E}{1+3\rho/2} =: \theta
  \quad\mbox{a.e. in }\Omega_T.
$$
In particular, $E=\theta+\frac32\rho\theta$. 
The bound for $(\theta_\tau)$ in $L^4(\Omega_T)$ (not uniform in $\delta$)
shows that the previous convergence holds in $L^\eta(\Omega_T)$ for any $\eta<4$. 
We deduce from the $L^2(\Omega_T)$ bounds for $\log\theta_\tau$ and
$\log\rho_\tau=\phi_\tau+\frac32\log\theta_\tau-\frac52$ that $\log\theta$ and
$\log\rho$ are integrable and thus, $\rho>0$, $\theta>0$ a.e.\ in $\Omega_T$.
Furthermore, $\phi_\tau\to\log\rho-\frac32\log\theta+\frac52=:\phi$ a.e.\ in $\Omega_T$
and, because of \eqref{bound.4}, weakly in $L^2(0,T;H^1(\Omega))$.

The previous bounds and the strong convergences of $(\rho_\tau)$ and $(\theta_\tau)$
allow us to pass to the limit $\tau\to 0$ in \eqref{ap5}--\eqref{ap6}. For this,
we observe that, by \eqref{rhoth},
$$
  \na(\rho_\tau\theta_\tau)\rightharpoonup\na(\rho\theta)\quad
	\mbox{weakly in }L^{4/3}(\Omega_T).
$$
Furthermore, by Lemma \ref{lem.higher},
$$
  \rho_\tau\theta_\tau^2\to \rho\theta^2\quad\mbox{strongly in }L^\eta(\Omega_T),\
	\eta<\frac32.
$$
The strong convergence of $(\theta_\tau)$ to $\theta$, the uniform bounds on
$(\theta_\tau)$, and the a.e.\ positivity of $\theta$ imply that
$$
  \theta_\tau^3\to\theta^3,\quad
	\theta_\tau^{-N}\log\theta_\tau \to \theta^{-N}\log\theta
	\quad\mbox{strongly in }L^1(\Omega_T).
$$
Finally, by Lemma \ref{lem.time},
\begin{align*}
  \tau^{-1}(\rho_\tau-\pi_\tau\rho_\tau)\rightharpoonup \pa_t\rho 
	&\quad\mbox{weakly in }L^{4/3}(0,T;W^{1,4}(\Omega)'), \\
	\tau^{-1}(E_\tau-\pi_\tau E_\tau)\rightharpoonup \pa_t E 
	&\quad\mbox{weakly in }L^{6/5}(0,T;W^{2,4}(\Omega)').
\end{align*}
Then \eqref{ap5}--\eqref{ap6} become in the limit $\tau\to 0$,
\begin{align}
  0 &= \int_0^T\langle\pa_t\rho,\psi_1\rangle dt
	+ \int_0^T\int_\Omega\na(\rho\theta)\cdot\na\psi_1 dxdt
	+ \delta\int_0^T\int_\Omega(\na\phi\cdot\na\psi_1+\phi\psi_1)dxdt, \label{ap7} \\
  0 &= \int_0^T\langle\pa_t E,\psi_2\rangle dt
	- \int_0^T\int_\Omega\bigg(\theta+\frac{5}{2}\rho\theta^2\bigg)\Delta\psi_2 dxdt 
	\label{ap8}\\
	&\phantom{xx}{}- \frac{\delta}{3}\int_0^T\int_\Omega\theta^3\Delta\psi_2 dxdt
	+ \delta\int_0^T\int_\Omega\theta^{-N}\log(\theta)\psi_2 dxdt \nonumber
\end{align}
for any test functions $\psi_1$, $\psi_2\in C_0^2(\Omega_T)$.


\subsection{Limit $\delta\to 0$}

In this subsection, we need some tools from mathematical fluid dynamics,
in particular the concept of renormalized solutions and the div-curl lemma.
In the following, we denote by $\overline{u_\delta}$ the weak or distributional
limit of a sequence $(u_\delta)$ whenever it exists. Let $(\rho_\delta,
E_\delta)$ be a weak solution to \eqref{ap7}--\eqref{ap8} and set
$\phi_\delta=\log(\rho_\delta/\theta_\delta^{3/2})+\frac52$,
$E_\delta=\theta_\delta+\frac32\rho_\delta\theta_\delta$.

{\em Step 1: Renormalized mass balance equation.}
We compute the renormalized form of \eqref{ap7}. Let $f\in C^2([0,\infty))
\cap L^\infty(0,\infty)$ satisfy $|f'(s)|\le C(1+s)^{-1}$ and 
$|f''(s)|\le C(1+s)^{-2}$ for $s\ge 0$. Furthermore, let $\xi\in C_0^\infty(\Omega_T)$.
Choosing $\psi_1=f'(\rho_\delta)\xi$ in \eqref{ap7}, we find that
\begin{align*}
  \int_0^T & \langle\pa_t f(\rho_\delta), \xi\rangle dt
  + \int_0^T\int_\Omega\big( f'(\rho_\delta)\na(\rho_\delta\theta_\delta)
  + \delta f'(\rho_\delta)\na\phi_\delta \big) \cdot\na\xi dx dt\\
  &= - \int_0^T\int_\Omega \big(\delta f'(\rho_\delta) \phi_\delta
	+ f''(\rho_\delta)\nabla\rho_\delta\cdot\nabla(\rho_\delta\theta_\delta)
  + \delta f''(\rho_\delta)\nabla\rho_\delta\cdot\nabla\phi_\delta\big)\xi dx dt.
\end{align*}
This computation can be made rigorous (such that $\psi_1$ is an admissible
test function) by using renormalization techniques; see, e.g.,
\cite[Section 10.18]{FeNo09}. The previous equation can be rewritten as
\begin{align}\label{rce.1}
  -\pa_t & f(\rho_\delta) 
	+ \diver\big(f'(\rho_\delta)\na(\rho_\delta\theta_\delta)
  + \delta f'(\rho_\delta)\na\phi_\delta \big) \\
  \nonumber 
  &= \delta f'(\rho_\delta) \phi_\delta 
	+ f''(\rho_\delta)\nabla\rho_\delta\cdot\nabla(\rho_\delta\theta_\delta)
  + \delta f''(\rho_\delta)\nabla\rho_\delta\cdot\nabla\phi_\delta 
	\quad\mbox{in }\mathcal{D}'(\Omega_T).
\end{align}

{\em Step 2: Application of the div-curl lemma.}
We apply the div-curl lemma to the vector fields
$$
  U_\delta = \big( f(\rho_\delta), -f'(\rho_\delta)
	\na(\rho_\delta\theta_\delta) - \delta f'(\rho_\delta)\na\phi_\delta
  \big),\quad V_\delta = \big( g(\theta_\delta) , 0,0,0 \big),
$$
where $f$ is as before and $g\in C^1([0,\infty))\cap L^\infty(0,\infty)$
satisfies $|g'(s)|\le C(1+s)^{-1}$ for $s>0$. 
We know from \eqref{bound.2} that $(\na\sqrt{\rho_\delta\theta_\delta})$
and $(\sqrt{\delta}\na\phi_\delta)$ are bounded in $L^2(\Omega_T)$ and from
Lemma \ref{lem.higher} that $(\sqrt{\rho_\delta\theta_\delta})$ is
bounded in $L^4(\Omega_T)$. Consequently,
$$
  f'(\rho_\delta)\na(\rho_\delta\theta_\delta) 
	+ \delta f'(\rho_\delta)\na\phi_\delta
  = 2f'(\rho_\delta)\sqrt{\rho_\delta\theta_\delta}
	\na\sqrt{\rho_\delta\theta_\delta}+\delta f'(\rho_\delta)\na\phi_\delta
$$
is uniformly bounded in $L^{4/3}(\Omega_T)$. Thus, $(U_\delta)$ is bounded
in $L^{4/3}(\Omega_T)$. 
Because of the properties of $g$, $(V_\delta)$ is trivially bounded in
$L^\infty(\Omega_T)$. 

The left-hand side of \eqref{rce.1} equals $-\diver_{(t,x)}U_\delta$. We wish to
bound the right-hand side of \eqref{rce.1}. For this, we observe that,
thanks to \eqref{bound.4}, the first term $\delta f'(\rho_\delta)\phi_\delta$ is 
uniformly bounded in $L^2(\Omega_T)$. We rewrite the second term as
\begin{equation}\label{aux}
  f''(\rho_\delta)\nabla\rho_\delta\cdot\nabla(\rho_\delta\theta_\delta) 
	= 4\rho_\delta f''(\rho_\delta)\sqrt{\theta_\delta}\nabla\sqrt{\rho_\delta}
	\cdot\nabla\sqrt{\rho_\delta\theta_\delta}.
\end{equation}
Since $\rho_\delta |f''(\rho_\delta)|\le C\rho_\delta/(1+\rho_\delta)^2\le C$
and $(\sqrt{\theta_\delta}\na\sqrt{\rho_\delta})$, 
$(\sqrt{\rho_\delta\theta_\delta})$ are bounded in $L^2(\Omega_T)$ by
\eqref{bound.2}, expression \eqref{aux} is bounded in $L^1(\Omega_T)$.
In order to bound the last term in \eqref{rce.1}, we observe that, by
\eqref{bound.3} and \eqref{bound.4},
$$
  \sqrt{\delta}\na\log\rho_\delta 
	= \sqrt{\delta}\na\phi_\delta + \frac32\sqrt{\delta}\na\log\theta_\delta
$$
is uniformly bounded in $L^2(\Omega_T)$. Then
$$
  \delta f''(\rho_\delta)\na\rho_\delta\cdot\na\phi_\delta
	= f''(\rho_\delta)\rho_\delta(\sqrt{\delta}\na\log\rho_\delta)\cdot
	(\sqrt{\delta}\na\phi_\delta)
$$
is uniformly bounded in $L^1(\Omega_T)$. We infer that the right-hand side of
\eqref{rce.1} and consequently also $-\diver_{(t,x)}U_\delta$
are uniformly bounded in $L^1(\Omega_T)$. By Sobolev's embedding, it follows
that $\diver_{(t,x)}U_\delta$ is relatively compact in $W^{-1,r}(\Omega_T)$
for some $r>1$. 

It follows from the uniform $L^2(\Omega_T)$ bound for $(\na\log\theta_\delta)$
(see \eqref{bound.3}) and $\theta_\delta |g'(\theta_\delta)|\le 
C\theta_\delta/(1+\theta_\delta)\le C$ that
$$
  \operatorname{curl}_{(t,x)}V_\delta
	= g'(\theta_\delta)\begin{pmatrix} 0 & (\na\theta_\delta)^T \\
	\na\theta_\delta & 0 \end{pmatrix}
	= \theta_\delta g'(\theta_\delta)\begin{pmatrix} 0 & (\na\log\theta_\delta)^T \\
	\na\log\theta_\delta & 0 \end{pmatrix}
$$
is uniformly bounded in $L^2(\Omega_T)$. By Sobolev's embedding, this expression is
relatively compact in $W^{-1,r}(\Omega_T;\R^{3\times 3})$ for some $r>1$.

The div-curl lemma \cite[Theorem 10.21]{FeNo09} implies that
$\overline{U_\delta\cdot V_\delta}=\overline{U_\delta}\cdot\overline{V_\delta}$
a.e.\ in $\Omega_T$, which means that
\begin{equation}\label{fg}
  \overline{f(\rho_\delta)g(\theta_\delta)} = \overline{f(\rho_\delta)}\,
	\overline{g(\theta_\delta)}\quad\mbox{a.e. in }\Omega_T
\end{equation}
for all $f\in C^2([0,\infty))\cap L^\infty(0,\infty)$ and $g\in C^1([0,\infty))
\cap L^\infty(0,\infty)$ satisfying $|f'(s)|\le C(1+s)^{-1}$,
$|f''(s)|\le C(1+s)^{-2}$, and $|g'(s)|\le C(1+s)^{-1}$ for $s>0$.

{\em Step 3: Proof of $\overline{\rho_\delta\theta_\delta}=\rho\theta$.}
We wish to relax the assumptions on the functions $f$ and $g$. To this end,
we introduce the truncation function $T_1\in C^2([0,\infty))$ by
$T_1(s)=s$ for $0\le s<1$, $T_1(s)=2$ for $s>3$, and $T_1$ is nondecreasing
and concave in $[0,\infty)$. Then we define $T_k(s)=k T_1(s/k)$ for $s>0$ and
$k\in\N$. It is possible to choose $f=T_k$ in \eqref{fg}. Together with
Fatou's lemma and the boundedness of $g$, we infer that
\begin{align*}
  \big\|\overline{\rho_\delta g(\theta_\delta)} 
	- \overline{\rho_\delta}~\overline{g(\theta_\delta)} \big\|_{L^1(\Omega_T)} 
	&= \big\|\overline{(\rho_\delta - T_k(\rho_\delta)) g(\theta_\delta)} 
	- \overline{(\rho_\delta - T_k(\rho_\delta))}~\overline{g(\theta_\delta)} 
	\big\|_{L^1(\Omega_T)}\\
  &\leq C\sup_{0<\delta<1}\int_{\Omega_T}|T_k(\rho_\delta)-\rho_\delta| dx dt.
\end{align*}
Furthermore, we deduce from \eqref{bound.1} that
$$
  \int_{\Omega_T}|T_k(\rho_\delta)-\rho_\delta| dx dt
	\leq C\int_{\{\rho_\delta\geq k\}}\rho_\delta dx dt
	\leq \frac{C}{\log k}\int_{\{\rho_\delta\geq k\}}\rho_\delta\log\rho_\delta dx dt
	\leq \frac{C}{\log k},
$$
such that we obtain for any $k\ge 2$,
$$
  \big\|\overline{\rho_\delta g(\theta_\delta)} 
	- \overline{\rho_\delta}~\overline{g(\theta_\delta)} \big\|_{L^1(\Omega_T)} 
	\le \frac{C}{\log k}.
$$
Then the limit $k\to\infty$ implies that
\begin{equation}\label{rho.g}
  \overline{\rho_\delta g(\theta_\delta)} = \rho\,\overline{g(\theta_\delta)}
	\quad\mbox{a.e. in }\Omega_T
\end{equation}
for any $g\in C^1([0,\infty))\cap L^\infty(0,\infty)$ satisfying  
$|g'(s)|\le C(1+s)^{-1}$ for $s>0$. We choose $g=T_k$ which leads to
\begin{equation}\label{rt.1}
  \overline{\rho_\delta \theta_\delta} - \rho\theta
	= \overline{\rho_\delta(\theta_\delta-T_k(\theta_\delta))}
	- \rho(\theta-\overline{T_k(\theta_\delta)}).
\end{equation}

We claim that both terms on the right-hand side converge to zero as $k\to\infty$.
Indeed, it follows from Fatou's lemma and the $L^1(\Omega_T)$ bound for
$(\rho_\delta\theta_\delta^2)$ from Lemma \ref{lem.higher} that
\begin{align*}
  \bigg\|\frac{\overline{\rho_\delta(\theta_\delta-T_k(\theta_\delta))}}{1+\rho}
	\bigg\|_{L^1(\Omega_T)}
	&\le \sup_{0<\delta<1}\int_{\Omega_T}\rho_\delta 
	|\theta_\delta - T_k(\theta_\delta)|dx dt
	\leq C\sup_{0<\delta<1}\int_{\{\theta_\delta>k\}}\rho_\delta\theta_\delta dx dt \\
  &\leq \frac{C}{k}\sup_{0<\delta<1}\int_{\{\theta_\delta>k\}}
	\rho_\delta\theta_\delta^2 dx dt\leq \frac{C}{k},
\end{align*}
while we deduce from Fatou's lemma and the $L^2(\Omega_T)$ bound for
$(\theta_\delta)$, again from Lemma \ref{lem.higher}, that
\begin{align*}
  \bigg\|\frac{\rho(\theta - \overline{T_k(\theta_\delta)})}{1+\rho}
	\bigg\|_{L^1(\Omega_T)} 
	&\leq \sup_{0<\delta<1}\int_{\Omega_T}|\theta_\delta - T_k(\theta_\delta)|dx dt
	\leq C\sup_{0<\delta<1}\int_{\{\theta_\delta>k\}}\theta_\delta dx dt \\
  &\leq \frac{C}{k}\sup_{0<\delta<1}\int_{\{\theta_\delta>k\}}\theta_\delta^2 dx dt
	\leq \frac{C}{k}.
\end{align*}
We infer from \eqref{rt.1} that for any $k\ge 1$,
$$
  \bigg\|\frac{\overline{\rho_\delta\theta_\delta}-\rho\theta}{1+\rho}
	\bigg\|_{L^1(\Omega_T)} \le \frac{C}{k},
$$
which implies, in the limit $k\to\infty$, that
\begin{equation}\label{rho.id}
  \overline{\rho_\delta\theta_\delta} = \rho\theta\quad\mbox{a.e. in }\Omega_T.
\end{equation}

{\em Step 4: Pointwise convergence of $(\theta_\delta)$.}
We prove via the Aubin--Lions lemma that $E_\delta=\theta_\delta+\frac32\rho_\delta
\theta_\delta$ is strongly convergent. We know from Lemma \ref{lem.grad} that
$(E_\delta)$ is bounded in $L^1(0,T;W^{1,1}(\Omega))$. For the time derivative
of $E_\delta$, we estimate \eqref{ap8} for $\psi_2\in C_0^\infty(\Omega_T)$:
\begin{align*}
  \bigg|\int_0^T\langle\pa_t E_\delta,\psi_2\rangle dt\bigg|
	&\le \bigg|\int_0^T\int_\Omega\bigg(\theta_\delta 
	+ \frac{5}{2}\rho_\delta\theta_\delta^2 + \frac{\delta}{3}\theta_\delta^3\bigg)
	\Delta\psi_2 dxdt\bigg| \\
	&\phantom{xx}{}+ \bigg|\delta\int_0^T\int_\Omega\theta_\delta^{-N}
	\log(\theta_\delta)\psi_2 dxdt\bigg| \\
  &\le C\big(\|\theta_\delta\|_{L^2(\Omega_T)}
	+ \|\rho_\delta\theta_\delta^2\|_{L^{3/2}(\Omega_T)}
	+ \delta\|\theta_\delta^3\|_{L^{4/3}(\Omega_T)}\big)
	\|\Delta\psi_2\|_{L^4(\Omega_T)} \\
	&\phantom{xx}{}+ C\big(1+\delta\|\theta_\delta^{-N}\log\theta_\delta
	\|_{L^{6/5}(\Omega_T)}\big)\|\psi_2\|_{L^6(\Omega_T)}.
\end{align*}
Taking into account estimate \eqref{est.L65} and again using Lemma \ref{lem.higher}, 
we infer that
$$
  \|\pa_t E_\delta\|_{L^{6/5}(0,T;W^{2,4}(\Omega)')} \le C.
$$

We apply the Aubin--Lions lemma to $(E_\delta)$ to obtain the existence
of a subsequence which is not relabeled such that, as $\delta\to 0$,
$(E_\delta)$ converges strongly in $L^\eta(\Omega_T)$ for $\eta<2$. 
Since $(1+\theta_\delta)^{-1}$ converges weakly in $L^\eta(\Omega_T)$ for
any $\eta<\infty$, we find that
\begin{equation}\label{over}
  \overline{\bigg(\theta_\delta + \frac32\rho_\delta\theta_\delta\bigg)
	(1+\theta_\delta)^{-1}} = \overline{\bigg(\theta_\delta 
	+ \frac32\rho_\delta\theta_\delta\bigg)}\,\overline{(1+\theta_\delta)^{-1}}
	\quad\mbox{a.e. in }\Omega_T.
\end{equation}
We choose $g(s)=s(1+s)^{-1}$ in \eqref{rho.g} and recall \eqref{rho.id}:
$$
  \overline{\rho_\delta\theta_\delta(1+\theta_\delta)^{-1}}
	= \rho\,\overline{\theta_\delta(1+\theta_\delta)^{-1}}, \quad
	\overline{\rho_\delta\theta_\delta} = \rho\theta\quad\mbox{a.e. in }\Omega_T.
$$
Using these expressions, we deduce from \eqref{over} that
\begin{align*}
  \bigg(1+\frac32\rho\bigg)&\overline{\theta_\delta(1+\theta_\delta)^{-1}}
	= \overline{\theta_\delta(1+\theta_\delta)^{-1}
	+\frac32\rho_\delta\theta_\delta(1+\theta_\delta)^{-1}} \\
	&= \overline{\bigg(\theta_\delta 
	+ \frac32\rho_\delta\theta_\delta\bigg)}\;\overline{(1+\theta_\delta)^{-1}} 
	= \overline{\theta_\delta}\;\overline{(1+\theta_\delta)^{-1}}
	+ \frac32\overline{\rho_\delta\theta_\delta}\;\overline{(1+\theta_\delta)^{-1}} \\
	&= \bigg(1+\frac32\rho\bigg)\theta\overline{(1+\theta_\delta)^{-1}}
	\quad\mbox{a.e. in }\Omega_T.
\end{align*}
This means that
$$
  \overline{\theta_\theta(1+\theta_\delta)^{-1}}
	= \theta\overline{(1+\theta_\delta)^{-1}}\quad\mbox{a.e. in }\Omega_T.
$$
We apply \cite[Theorem 10.19]{FeNo09} to the strictly decreasing 
function $s\mapsto(1+s)^{-1}$ for $s\ge 0$ to conclude that 
$$
  \overline{(1+\theta_\delta)^{-1}} = (1+\theta)^{-1}\quad\mbox{a.e. in }\Omega_T.
$$
The strict convexity of $s\mapsto(1+s)^{-1}$ then implies, by
\cite[Theorem 10.20]{FeNo09}, that $\theta_\delta\to\theta$ a.e.\ in $\Omega_T$. 
We deduce from the $L^2(\Omega_T)$ bound for $(\theta_\delta)$ from Lemma 
\ref{lem.higher} that this convergence is in fact strong in $L^1(\Omega_T)$.

{\em Step 5: Limit $\delta\to 0$ in equations \eqref{ap7}--\eqref{ap8}.}
We know from \eqref{rhoth} that $(\na(\rho_\delta\theta_\delta))$ is bounded
in $L^{4/3}(\Omega_T)$. Thus, up to a subsequence,
$\na(\rho_\delta\theta_\delta)\rightharpoonup\zeta_1$ weakly in $L^{4/3}(\Omega_T)$
for some $\zeta_1\in L^{4/3}(\Omega_T)$.
Since $\rho_\delta\theta_\delta\rightharpoonup\rho\theta$ weakly in
$L^1(\Omega_T)$, by \eqref{rho.id}, we infer that $\zeta_1=\na(\rho\theta)$, i.e.
\begin{equation}\label{conv1}
  \na(\rho_\delta\theta_\delta)\rightharpoonup\na(\rho\theta)
	\quad\mbox{weakly in }L^{4/3}(\Omega_T).
\end{equation}

We know from Lemma \ref{lem.higher} that $(\rho_\delta\theta_\delta^2)$ is
bounded in $L^{3/2}(\Omega_T)$, so that up to a subsequence,
$\rho_\delta\theta_\delta^2\to\zeta_2$ weakly in $L^{3/2}(\Omega_T)$.
We deduce from the strong convergence of $(\theta_\delta)$ and the boundedness
of $s\mapsto(1+s^2)^{-1}$ that $(1+\theta_\delta^2)^{-1}\to(1+\theta^2)^{-1}$
strongly in $L^\eta(\Omega_T)$ for any $\eta<\infty$. Therefore,
$$
  \frac{\rho_\delta\theta_\delta^2}{1+\theta_\delta^2}
	\rightharpoonup\frac{\zeta_2}{1+\theta^2}\quad\mbox{weakly in }L^1(\Omega_T).
$$
An application of \eqref{rho.g} with $g(s)=s^2(1+s^2)^{-1}$ together with the
strong convergence of $(\theta_\delta)$ leads to
$$
  \frac{\rho_\delta\theta_\delta^2}{1+\theta_\delta^2}
	\rightharpoonup\frac{\rho\theta^2}{1+\theta^2}\quad\mbox{weakly in }L^1(\Omega_T).
$$
Hence, $\zeta_2=\rho\theta^2$ a.e.\ in $\Omega_T$ and
\begin{equation}\label{conv2}
  \rho_\delta\theta_\delta^2\rightharpoonup \rho\theta^2\quad
	\mbox{weakly in }L^{3/2}(\Omega_T).
\end{equation}
Furthermore, it follows from \eqref{bound.4}, Lemma \ref{lem.higher}, and
\eqref{est.L65} that
\begin{align}
  \delta\phi_\delta\to 0 &\quad\mbox{strongly in }L^2(0,T;H^1(\Omega)), \nonumber \\
	\delta\theta_\delta^3\to 0 &\quad\mbox{strongly in }L^{4/3}(\Omega_T), 
	\label{conv3} \\
  \delta\theta_\delta^{-N}\log\theta_\delta\to 0
	&\quad\mbox{strongly in }L^{6/5}(\Omega_T). \nonumber
\end{align}
For any $\psi_1\in L^4(0,T;W^{1,4}(\Omega))$, we have
\begin{align*}
  \bigg|\int_0^T\langle\pa_t\rho_\delta,\psi_1\rangle dt\bigg|
	&\le \frac32\|\na(\rho_\delta\theta_\delta)\|_{L^{4/3}(\Omega_T)}
	\|\na\psi_1\|_{L^4(\Omega_T)} \\
	&\phantom{xx}{}+ \delta\|\phi_\delta\|_{L^2(0,T;H^1(\Omega))}
	\|\psi_1\|_{L^2(0,T;H^1(\Omega))}\le C.
\end{align*}
Hence, up to subsequences,
\begin{equation}\label{conv4}
\begin{aligned}
  \pa_t\rho_\delta\rightharpoonup\pa_t\rho &\quad\mbox{weakly in }
	L^{4/3}(0,T;W^{1,4}(\Omega)'), \\
  \pa_t E_\delta\rightharpoonup\pa_t E &\quad\mbox{weakly in }
	L^{6/5}(0,T;W^{2,4}(\Omega)').
\end{aligned}
\end{equation}
We deduce from the bound for $(\log\theta_\delta)$ in $L^\infty(0,T;L^1(\Omega))$
that $\theta>0$ a.e.\ in $\Omega_T$. 

We claim that $(\rho_\delta)$ also converges strongly. Indeed, the a.e.\ 
convergence of $(E_\delta)$ and $(\theta_\delta)$ imply that
$\rho_\delta=\frac23(E_\delta/\theta_\delta-1)\to \rho$ a.e.\ in $\Omega_T$.
The $L^\infty(0,T;L^1(\Omega))$ bound for $(\rho_\delta\log\rho_\delta)$
from \eqref{bound.1} shows that $(\rho_\delta)$ is equi-integrable, and
together with its a.e.\ convergence, we conclude from the de la Vall\'ee--Poussin
theorem \cite[Chap.~8, Sect.~1.7, Corollary 1.3]{EkTe76} that
$$
  \rho_\delta\to\rho\quad\mbox{strongly in }L^1(\Omega_T).
$$
The positivity of $\rho_\delta$ implies that $\rho\ge 0$ a.e.\ in $\Omega_T$.
Note, however, that we cannot conclude that $\rho>0$ a.e., since the control
on $\phi_\delta$ is now lost.

Convergences \eqref{conv1}--\eqref{conv4} allow us to perform the limit
$\delta\to 0$ in \eqref{ap7}--\eqref{ap8} showing that $(\rho,\theta)$
solves \eqref{weak.1}--\eqref{weak.2}. Theorem \ref{thm.ex} is proved.


\end{document}